\documentclass[11pt]{amsart}

\usepackage{graphicx}
\usepackage{amssymb}
\date{}

\usepackage[margin=1in]{geometry}  
\usepackage{graphicx}              
\usepackage{amsmath}               
\usepackage{amsfonts}              
\usepackage{amsthm}                
\usepackage{amsmath, amssymb}
\usepackage[all]{xy}
\usepackage{enumerate}
\usepackage{color}

\newtheorem{thm}{Theorem}[section]
\newtheorem{lem}[thm]{Lemma}
\newtheorem{prop}[thm]{Proposition}
\newtheorem{cor}[thm]{Corollary}

\theoremstyle{definition}
\newtheorem{remark}[thm]{Remark}

\theoremstyle{definition}
\newtheorem{notation}[thm]{Notation}

\theoremstyle{definition}
\newtheorem{defn}[thm]{Definition}

\theoremstyle{definition}
\newtheorem{ex}[thm]{Example}

\newcommand\smvee{\raise0.0ex\hbox{$\scriptscriptstyle\vee$}}

\usepackage{MnSymbol}
\usepackage{wasysym}
\usepackage{amsfonts}
\usepackage{epstopdf}
\usepackage{pgf, tikz}
\usetikzlibrary{arrows}
\usepackage{mathtools}
\usepackage{ dsfont }
\usetikzlibrary{calc,positioning,intersections, arrows, automata}
\usetikzlibrary{calc,positioning,intersections, arrows, automata,decorations.pathreplacing,shapes.misc}
\usepackage{graphicx}
\usepackage{xcolor}
\usepackage{amsmath}
\usepackage{amsfonts}
\usepackage{MnSymbol}
\usepackage{epstopdf}
\usepackage{pgf, tikz}
\usetikzlibrary{arrows, automata}
\usepackage{amsthm,amsmath,amssymb}

\title{Minimal Length Maximal Green Sequences and Triangulations of Polygons}

\author{E.Cormier} 
\address{Department of Mathematics, Bowdoin College, Brunswick, ME 04011, USA}
\email{ecormier@bowdoin.edu}
 
 \thanks{This research was carried out at the University of Connecticut 2015 math REU funded by NSF under DMS-1262929.  The fourth author was also supported by the NSF CAREER grant DMS-1254567.}

\author{ P. Dillery}
\address{Department of Mathematics, University of Virginia, Charlottesville, VA 22904-4137, USA}
\email{ped7pc@virginia.edu}

\author{J. Resh} 
\address{Department of Mathematics, Roger Williams University, Bristol, RI 02809, USA}
\email{jresh123@g.rwu.edu}

\author{K. Serhiyenko}
\address{Department of Mathematics, University of Connecticut, Storrs, CT 06269-3009, USA}
\email{khrystyna.nechyporenko@uconn.edu}

\author{ J. Whelan}
\address{Department of Mathematics, Vassar College, Poughkeepsie, NY 12604, USA}
\email{jowhelan@vassar.edu}

\begin{document}
\maketitle

\begin{abstract}
We use combinatorics of quivers and the corresponding surfaces to study maximal green sequences of minimal length for quivers of type $\mathbb{A}$. We prove that such sequences have length $n+t$, where $n$ is the number of vertices and $t$ is the number of 3-cycles in the quiver.  Moreover, we develop a procedure that yields these minimal length maximal green sequences.    
\end{abstract}

 \section {Introduction}
 The study of maximal green sequences is motivated by string theory, in particular Donaldson-Thomas invariants and the BPS spectrum, see \cite{ACCERV, MR}. The term maximal green sequence (MGS) was first introduced by Keller in \cite{K}, as a tool to compute the refined DT-invariants and quantum dilogarithm identities.  The definition of MGS's is purely combinatorial.  It involves local transformations of quivers, also known as directed graphs.   Moreover, these sequences are closely related to the study of Fomin-Zelevinsky's cluster algebras and various topics in the representation theory of algebras, see \cite{BY} for an overview.  
 
 A significant body of research on MGS's has been devoted to determining both the existence of MGS's for particular classes of quivers, as well as investigating the number of MGS's pertaining to a given quiver.  Notable results in this line of inquiry appear in \cite{BHIT} and \cite{BDP}, where the authors show that a quiver mutation-equivalent to an acyclic tame type quiver admits only finitely many maximal green sequences.  The same statement holds for acyclic quivers with at most three vertices.  

In this paper, we are interested in studying quivers arising from triangulations of surfaces, see \cite{FST} for details.  For almost all surfaces, in \cite{ACCERV} the authors construct specific triangulations and show that MGS's exist for quivers obtained from these triangulations.  Some of the excluded cases were explored further in \cite{B, BM, CLS}.  It is interesting to note that the existence of MGS's is not preserved under mutation as shown in \cite{M}.  

It follows from \cite{FZ2} that triangulations of unpunctured disks yield quivers that are mutation-equivalent to type $\mathbb{A}$ Dynkin diagrams.  Moreover, maximal green sequences always exist for such quivers.  Further developments for type $\mathbb{A}$ quivers were done by Garver and Musiker in \cite{GM}.  Using a combinatorial approach, the authors describe a procedure that produces an MGS for such quivers. However, in general this procedure does not yield a sequence of minimal length.  Recently, it was shown in \cite{GMc} and \cite{Ka}, using representation theoretic techniques, that for a given quiver the lengths of MGS's form a connected interval.  This property was first conjectured in \cite{BDP}.   

It is known that for acyclic quivers the length of a shortest MGS coincides with the number of vertices in the quiver, see \cite{BDP}.  However, this does not hold for quivers with oriented cycles.  Moreover, there is no known formula for the length of a shortest MGS for such quivers.  In this paper we find the minimal length of MGS's for quivers arising from triangulations of unpunctured disks.  We use both the combinatorics of quivers and surfaces to establish the following result.
  
\begin{thm}
The minimal length of a maximal green sequence for a quiver of type $\mathbb{A}$ is $n+t$, where $n$ is the number of vertices in the quiver and $t$ is the number of 3-cycles.
\end{thm}

Furthermore, in Theorem \ref{6.4} we describe an explicit procedure that yields MGS's of minimal length for quivers discussed above.  It relies on combining the procedures for acyclic quivers and for quivers consisting entirely of connected 3-cycles, see Proposition \ref{4.5} and Theorem \ref{5.9} respectively.   

The paper is organized as follows. In Section 2, we review the definition of a maximal green sequence and state the relevant results.  Then in the following section, we discuss the construction of quivers arising from triangulations of polygons. In Section 4, we describe a procedure that produces minimal length MGS's for acyclic quivers, and in Section 5, we develop a procedure that produces MGS's for quivers composed entirely of joined 3-cycles.  Combining these results in Section 6, we obtain a method that yields an MGS for arbitrary quivers of type $\mathbb{A}$.  In the last section, using properties of the corresponding triangulations, we prove the formula for the length of the shortest MGS's and show that the given procedure results in MGS's of desired length. 


\section{Maximal Green Sequences} 

We begin by establishing some terminology.  

\begin{defn} 
A \textit{quiver} $Q=(Q_0,Q_1)$  is a directed graph where $Q_0$ is a set of vertices and $Q_1$ is a set of arrows. 
\end{defn}

Throughout this paper we always assume that each quiver $Q$ is finite, meaning both $Q_0$ and $Q_1$ are finite sets.   For simplicity, we label the vertices of $Q$ consecutively 1 through $n$.  A vertex $i\in Q_0$ is called a \emph{source} if there are no arrows ending at $i$, and it is called a \emph{sink} if there are no arrows starting at $i$.  

We only consider quivers without loops ($\xymatrix{\bullet \ar@(ru,rd)}$ \text{    }\text{        }\text{  }\text{    }) and oriented 2-cycles ($\xymatrix{\bullet \ar@/^/[r] &\bullet \ar@/^/[l]}$). The following construction due to Fomin and Zelevinsky describes an operation on such quivers known as mutation, see \cite{FZ}.   

\begin{defn} 
Given a vertex $i \in Q_0$, the \textit{mutation} of a quiver $Q$ at $i$, denoted $\mu_i(Q)$, is a new quiver formed by applying the following steps to $Q$.

1. For any pair of arrows $h \rightarrow i \rightarrow j$ in $Q$, add a new arrow $h \rightarrow j$.

2. Reverse all arrows incident to $i$.

3. Remove a maximal collection of oriented 2-cycles.

\end{defn}

Observe that this is a local transformation of a quiver that preserves the number of vertices.  In addition, the mutation procedure always yields a quiver without loops and oriented 2-cycles.   Thus, the resulting quiver can be mutated further.




\begin{ex} 
Consider the quiver $Q$ below composed of three vertices and three arrows arranged in a cyclic orientation.  On the right we see the resulting quiver $\mu_1 (Q)$.  Observe that in step 1 of the mutation procedure we add a new arrow from vertex 3 to vertex 2. However, this creates an oriented 2-cycle, which is removed in step 3.  

$$\xymatrix@C=15pt @R=15pt{& 3\ar[dl]&&&&&&& 3 \\ 1 \ar[rr]&& 2\ar[ul]&&\ar[r]^{\mu_1}&&& 1\ar[ur] &&2 \ar[ll]}$$
 \end{ex}

\begin{defn} 
Given a quiver $Q$, construct the corresponding \emph{framed quiver} $\hat{Q}=(\hat{Q}_0, \hat{Q}_1)$ and the \emph{coframed quiver} $\overset{\smvee}{Q}=(\overset{\smvee}{Q}_0, \overset{\smvee}{Q}_1)$ by adding a set of \textit{frozen vertices} $Q_0'=\{1',2',\dots,n' \}$  such that 

$$\hat{Q}_0=\overset{\smvee}{Q}_0=Q_0 \cup Q_0' $$

and a set of arrows such that 

$$ \hat{Q}_1 = Q_1 \cup \{i \rightarrow i'\mid i \in Q_0\}$$
$$ \overset{\smvee}{Q}_1 = Q_1 \cup \{i \leftarrow i'\mid i \in Q_0 \}.$$

\end{defn}

\begin{ex} \label{1}

Let $Q$ be the quiver $1\longrightarrow 2$.  Below we show the corresponding framed quiver and the coframed quiver.

$$\xymatrix@C=25pt @R=15pt{&1'& 2' &&&&&1'\ar[d]&2'\ar[d] \\ \hat{Q}: &1\ar[r]\ar[u]&2\ar[u]&&&&\overset{\smvee}{Q}:&1\ar[r]&2}$$

\end{ex}

It can be seen that a quiver mutation is an involution, meaning $\mu_i (\mu_i (Q)) =Q$ for any vertex $i$.  This gives rise to the following equivalence relation.  We say $Q'$ is \emph{mutation equivalent} to $\hat{Q}$, denoted by $Q' \sim \hat{Q}$, if there exists a finite sequence of mutations at the non-frozen vertices $\mu_{i_j}\dots\mu_{i_2} \mu_{i_1}$ transforming $\hat{Q}$ into $Q'$. In other words $\mu_{i_j}\dots\mu_{i_2} \mu_{i_1}(\hat{Q})=Q'$.  The equivalence class of $\hat{Q}$ is called its \emph{mutation class} and denoted by Mut$(\hat{Q})$.  Note that this definition also makes sense if we replace the framed quiver $\hat{Q}$ by an ordinary quiver $Q$ and allow mutations at all vertices.

Moreover, we will only consider quivers up to isomorphisms.  In the case of quivers with frozen vertices this means the following.  Two quivers $Q'$ and $Q''$ with the same set of frozen vertices are \emph{isomorphic} if there exists a quiver isomorphism $\phi: Q' \to Q''$ that fixes the frozen vertices.

\begin{defn} 
A non-frozen vertex $i$ is called \textit{green} if there are no arrows from frozen vertices into $i$. A non-frozen vertex $j$ is called \textit{red} if there are no arrows from $j$ into frozen vertices. 
\end{defn}

Observe, that by construction the non-frozen vertices of $\hat{Q}$ are green, and the non-frozen vertices of $\overset{\smvee}{Q}$ are red for any quiver $Q$.  It is interesting to see how the color of the vertices is affected by mutations.  Consider the following results.

\begin{thm}\cite{BDP}\label{1.1}
Given any $Q' \in \textup{Mut}(\hat{Q})$, every non-frozen vertex of $Q'$ is either red or green. 
\end{thm}

\begin{prop} \cite{BDP} \label{2.8}
Let $Q' \in \textup{Mut}(\hat{Q})$.\\
\indent (1) If all the non-frozen vertices of $Q'$ are green, then $Q' \cong \hat{Q}$.\\
\indent (2) If all the non-frozen vertices of $Q'$ are red, then $Q' \cong \overset{\smvee}{Q}$.
\end{prop}

This gives rise to the definition of a maximal green sequence.

\begin{defn} 
A \textit{green sequence} for a quiver $Q$ is a finite sequence of mutations $\mu_{i_j}\dots \mu_{i_2}\mu_{i_1}$ of $\hat{Q}$ such that each consecutive mutation is performed in a green vertex.  This means, that for all $h = 1, \dots , j$, the vertex $i_h$ is green in the quiver $\mu_{i_{h-1}}\dots\mu_{i_2}\mu_{i_1}(\hat{Q})$.   A \textit{maximal green sequence} (MGS) for $Q$ is a green sequence that transforms $\hat{Q}$ into a quiver where every non-frozen vertex is red.  The \emph{length} of an MGS is the number of mutations appearing in the sequence.  
\end{defn}

Observe that by Proposition \ref{2.8} a maximal green sequence takes $\hat{Q}$ to a quiver isomorphic to $\overset{\smvee}{Q}$.  Moreover, we can say more in the situation when $Q$ is an acyclic quiver.  

\begin{defn}
An \emph{admissible sequence of sources} in $Q$ is a total ordering $(i_1, \dots , i_n)$ of all vertices in $Q$ such that:  
\begin{enumerate}[(i)]
\item $i_1$ is a source in $Q$, and 
\item $i_j$ is a source in $\mu_{i_{j-1}}\mu_{i_{j-2}}\cdots \,\mu_{i_1} (Q)$. 
\end{enumerate}
\end{defn}

It is well known that for acyclic quivers there exists an admissible sequence of sources.  In this case we have the following result.  

\begin{lem}\cite{BDP}\label{2.9}
Let $Q$ be an acyclic quiver.  Then any admissible sequence of sources $(i_1, \dots, i_n)$ in $Q$ yields a maximal green sequence $\mu_{i_n}\cdots \mu_{i_1}$ for $Q$.  
\end{lem}

\begin{ex}  
The quiver $Q$ defined in Example \ref{1} has an MGS $\mu_2 \mu_1$.  Observe that each mutation is a source mutation.  
$$\xymatrix@C=25pt @R=15pt{1'& 2' &&&1'\ar[d]&2'&&&1'\ar[d]&2'\ar[d] \\ 1\ar[r]\ar[u]&2\ar[u]&\ar[r]^{\mu_1}&& 1&2\ar[l] \ar[u]&\ar[r]^{\mu_2}&&1\ar[r]&2}$$
\end{ex}

It is important to note that for most quivers that admit a maximal green sequence, such a sequence is not unique.  Moreover, maximal green sequences vary in length, and in general there are many maximal green sequences of the same length for a given quiver.  

\begin{remark} \label{2.10}
Although red vertices can become green by mutating at the adjacent green vertices, if follows from \cite{BDP} that a green vertex can only become red if we mutate at that vertex itself. Thus for a quiver $Q$ with $n$ vertices, any MGS for $Q$ will always consist of at least $n$ mutations.  
\end{remark}

Next, consider the following lemma, which will be used in the proof of the main theorem.  We note that unlike the rest of the results, it works for a generic quiver $Q$.  

\begin{lem}\label{2.14}
Let $Q' \in \textup{Mut} (\hat{Q})$ for some quiver $Q$, such that the following conditions hold. 

\begin{enumerate}[(i)]
\item $\mathcal{C}$ is a full subquiver of $Q$ and $\hat{\mathcal{C}}$ is a full subquiver of $Q'$.  
\item $Q'$ is composed of two full subquivers $\hat{\mathcal{C}}$ and $\mathcal{D}$ connected by a single arrow $i\rightarrow j$, where $i\in\mathcal{C}_0$ and $j\in\mathcal{D}_0$.   
\end{enumerate} 

\noindent Let $\mu_{\mathcal{C}}$ be a maximal green sequence for $\mathcal{C}$, then 

\begin{enumerate}[1.]
\item all vertices of $\mathcal{C}$ are red in $\mu_{\mathcal{C}} (Q')$;
 \item $\mu_{\mathcal{C}}(Q')$ is composed of two full subquivers $\mu_{\mathcal{C}} (\hat{\mathcal{C}})$ and $\mathcal{D}$ connected by a single arrow $j\rightarrow x$,  where $x$ is a unique vertex of $\mathcal{C}_0$ with an arrow $i' \rightarrow x$ starting at the frozen vertex $i'$.  
\end{enumerate} 

$$ \begin{tikzpicture}[scale = .7][
> = stealth, 
            shorten > = 4pt, ]
            
\node[draw=none, fill=none] at (7.7,0.4) {$\hat{\mathcal{C}}$};
\node[draw, shape=circle, fill=black, scale=0.5] at (10.5,0) { };
\node[draw=none, fill=none] at (10.5,-0.5) {$i$};
\node[draw, shape=circle, fill=black, scale=0.5] at (12.5,0) { };
\node[draw=none, fill=none] at (12.5,-0.5) {$j$};
\node[draw=none, fill=none] at (15.7,0.4) {$\mathcal{D}$};

\draw [black] plot [smooth cycle] coordinates {(4.5,0) (6.5,2) (10.5,2) (10.5,0) (8.5,-2)};
\draw [black] plot [smooth cycle] coordinates {(12.5,0) (14.5,2) (18.5,2) (18.5,0) (16.5,-2)};

\draw [->] (10.6,0) -- node { } (12.35,0);

\end{tikzpicture}$$ 

\end{lem}

\begin{proof}
Proof of part 1. We claim that the presence of vertex $j$ connected to $\hat{\mathcal{C}}$ cannot add any new arrows between vertices of $\hat{\mathcal{C}}$ strictly through mutations in the vertices of $\mathcal{C}$. Observe, that in order to create an arrow between vertices $x$ and $y$ in $\hat{\mathcal{C}}$ it is necessary to have an arrow $x \rightarrow z$ and an arrow $z \rightarrow y$ before a mutation at some vertex $z$.  Since only vertices of $\mathcal{C}$ are being mutated, the only arrows that can be created, due to the presence of the vertex $j$, are those between $j$ and vertices of $\hat{\mathcal{C}}$. 

Next we show that the existence of $j$ does not interfere with the green mutation sequence $\mu_{\mathcal{C}}$.  Given a red vertex $x$ in $\mathcal{C},$ there must be a frozen vertex $y'$ in $\hat{\mathcal{C}}$ with an arrow $y'\rightarrow x$.  In order for $x$ to become green, there must be a mutation at a vertex $z$ in $\mathcal{C}$ such that there exist arrows $x\rightarrow z$ and $z\rightarrow y'$ prior to $\mu_z$.  Since the presence of $j$ cannot add any new arrows between vertices in $\hat{\mathcal{C}}$ during the mutation sequence $\mu_{\mathcal{C}}$, the existence of $j$ cannot cause any vertices in $\mathcal{C}$ to become green.  This implies that $\mu_{\mathcal{C}}$ makes all vertices in $\mathcal{C}$ red, as desired. 

Proof of part 2.  By assumption there exists a vertex $i$ in $\mathcal{C}$ with $i \rightarrow j$ for some $j\in\mathcal{D}_0$.  Note that the frozen vertex $i'$ and the vertex $j$ are in identical configurations relative to the remaining vertices of $\hat{\mathcal{C}}$ in $Q'$.  In particular,  $i'$ and $j$ are each connected to a unique vertex $i$ in $\hat{\mathcal{C}}$ by a single arrow starting at $i$.  Since, neither $i'$ nor $j$ will be mutated during $\mu_{\mathcal{C}}$ they will behave in the exact same way.  This means that throughout the mutation sequence $\mu_{\mathcal{C}}$, for any vertex $z \in \mathcal{C}$ there exists $i'\rightarrow z$ if and only if there exists $j\rightarrow z$, and similarly there exists $z\rightarrow i'$ if and only if there exists $z\rightarrow j$.  Therefore, since in $\mu_{\mathcal{C}}(Q')$ all vertices in $\mathcal{C}$ are red by part 1, the only arrows connecting $j$ and vertices in $\hat{\mathcal{C}}$ will be coming from $j$ into $\hat{\mathcal{C}}$.   Moreover, in the proof of part 1 we showed that the presence of vertex $j$ does not add any new arrows between vertices of $\hat{\mathcal{C}}$  during the mutation sequence $\mu_{\mathcal{C}}$.  This implies that in $\mu_{\mathcal{C}}(Q')$ there is one arrow starting at $j$ and ending at a vertex $x\in \mathcal{C}_0$, where $x$ is a unique vertex with an arrow $i' \rightarrow x$ from the frozen vertex $i'$.  

Now it suffices to show that in $Q'$ given a vertex $y$ in $\mathcal{D}$ connected to $j$ by an arrow, there are no arrows between $y$ and $\hat{\mathcal{C}}$ in the resulting quiver $\mu_{\mathcal{C}}(Q')$.  In order for such arrow to exist there must be a mutation at some vertex $z$ connected to both $y$ and some vertex in $\hat{\mathcal{C}}$.  However, throughout the mutation sequence $\mu_{\mathcal{C}}$ we can see that $y$ will not be connected to any vertex of $\mathcal{C}$, because in $Q'$ it is not connected to any such vertex.  Therefore, no vertices of $\mathcal{D}$ besides $j$ are connected to $\hat{\mathcal{C}}$ in $\mu_{\mathcal{C}}(Q')$.  Since the mutation sequence $\mu_{\mathcal{C}}$ only involves mutations at vertices of $\mathcal{C}$, the subquiver $\mathcal{D}$ is unaffected by $\mu_{\mathcal{C}}$.  Hence, $\mathcal{D}$ remains a full subquiver of $\mu_{\mathcal{C}}(Q')$.
\end{proof}

\begin{remark}\label{2.15}
The above lemma can be generalized, replacing a single component $\mathcal{D}$ by a set of components $\{\mathcal{D}_k\}$ each connected to a given $\hat{\mathcal{C}}$ by an arrow $i_k \rightarrow j_k$, where $i_k \in \mathcal{C}$, $j_k\in\mathcal{D}_k$, and all $i_k$ are distinct.  Observe that in this case it suffices to show that there are no arrows between different $\mathcal{D}_k$ throughout the mutation sequence $\mu_{\mathcal{C}}$.  This holds because in particular, there are no arrows between different $j_k$, as each of them behaves analogously to the corresponding frozen vertex ${i'_k}$, and  it is well-known that there are no arrows between frozen vertices for any quiver in $\text{Mut}(\hat{Q})$.  
\end{remark} 

\section{Quivers Arising from Triangulations of Polygons}

A particularly nice class of quivers are those arising from triangulations of surfaces.  Given a triangulation $T$ of a marked oriented Riemann surface $(S,M)$ with boundary,  in \cite{FST} the authors describe a construction that produces the corresponding quiver $Q_T$.  Moreover, mutations of $Q_T$ have a concrete geometric realization.  This enables one to study mutations of quivers coming from surfaces in this alternative setting.  

In this paper, we study the case when $(S,M)$ is a disk with $m$ marked points on the boundary of $S$. Observe that this surface is homotopic to a polygon with $m$ edges, and in the literature these surfaces are used interchangeably.  From now on $(S,M)$ will always denote such a surface oriented counterclockwise. For brevity, we only describe the correspondence between quivers and triangulations of regular polygons.  See \cite{FST} for the general situation.   

\begin{defn} 
Given a regular polygon $(S,M)$ an {\emph{arc}} $\gamma$ in $S$ is a diagonal of the polygon.  
\end{defn}

\begin{defn} 
A \textit{triangulation} $T$ of $(S,M)$  is a maximal collection of arcs that do not intersect in the interior of $S$. 
\end{defn}

\begin{remark} \label{3.0}
Any given triangulation of an $(n+3)$-gon is composed of $n$ arcs, which in turn divide the surface into $n+1$ distinct triangles.  
\end{remark}

Following \cite{FST} we describe a method to obtain a quiver from a given triangulation.  

\begin{defn} Given a triangulation $T$ of a polygon $S$, form the corresponding quiver $Q_T$ by applying the following steps. 

\begin{enumerate}[1.]

\item Label the arcs consecutively $1$ through $n$. The arcs correspond to vertices $1$ through $n$ in the quiver $Q_T$.

\item For any pair of arcs $i$ and $j$ belonging to the same triangle $\Delta$ in $T$, add an arrow from $i$ to $j$ in the quiver if $j$ follows $i$ in $\Delta$ with respect to the counterclockwise orientation of $S$.  
\end{enumerate}
\end{defn}

To emphasize the difference between the boundary $\partial S$ of $S$ and arcs in $S$, we represent $\partial S$ by smooth curves, and the arcs in $S$ by straight line segments.  

\begin{ex} 
Consider the following triangulation of an 11-gon.  On the left we construct the corresponding quiver.

\hspace*{\fill} \\

 \includegraphics[scale=0.4]{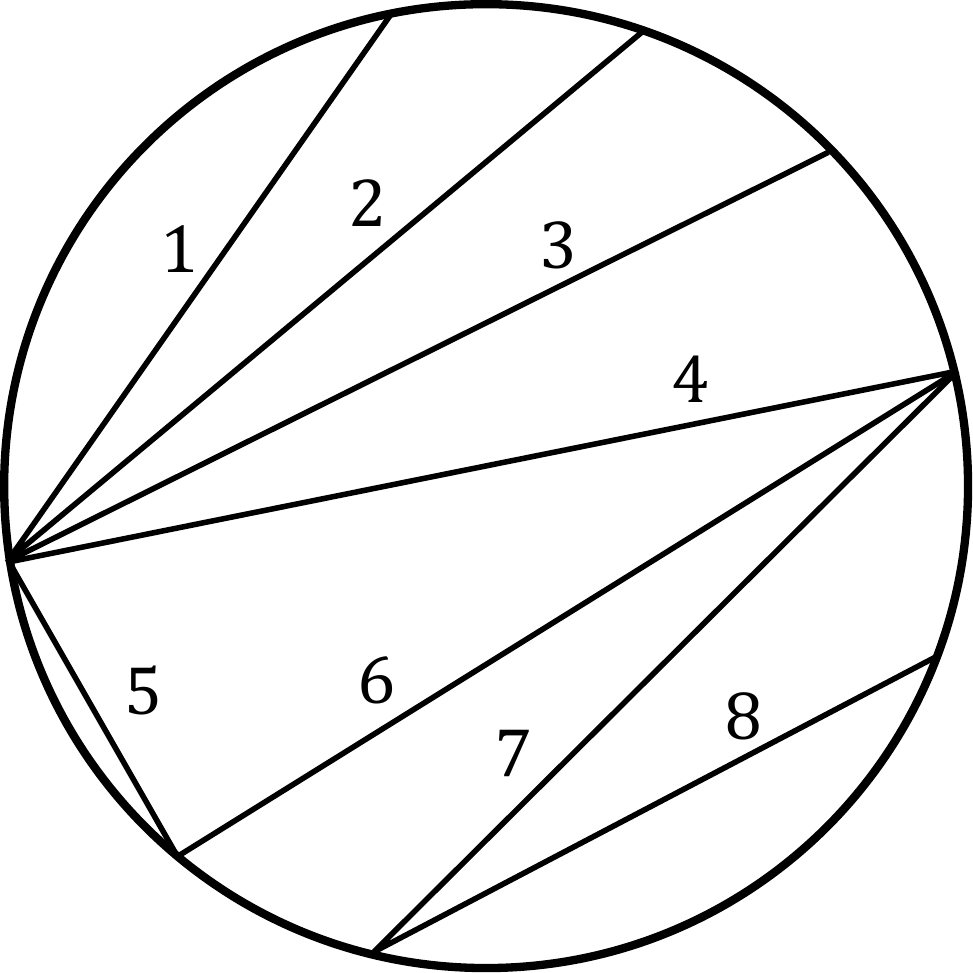} 
$
\hspace{2cm} \vspace{-1cm}
\begin{tikzpicture}[scale=0.8][
            > = stealth, 
            shorten > = 1pt, 
            auto,
            shorten < = 1pt, 
            auto,
            node distance = 1.35cm, 
            semithick 
        ]

        \tikzstyle{every state}=[
            draw = black,
            thick,
            fill = white,
            minimum size = 3mm
        ]

        \node[draw=none, fill=none] (1) {$1$};      
        \node[draw=none, fill=none] (2) [right of= 1] {$2$};
        \node[draw=none, fill=none] (3) [right of= 2] {$3$};
        \node[draw=none, fill=none] (4) [right of=3] {$4$};
        \node[draw=none, fill=none] (5) [right of=4] {$5$};
        \node[draw=none, fill=none] (6) [right of=5] {$6$};
        \node[draw=none, fill=none] (7) [right of=6] {$7$};
        \node[draw=none, fill=none] (8) [right of=7] {$8$};

        \path[->] (1) edge node { } (2);
        \path[->] (2) edge node { } (3);
        \path[->] (3) edge node { } (4);
        \path[->] (4) edge node { } (5);
        \path[->] (5) edge node { } (6);
        \path[->] (7) edge node { } (6);
        \path[->] (7) edge node { } (8);
        \path[->] (6) edge [bend right] node { } (4);

   \end{tikzpicture}  \\ \\$

\end{ex}

\begin{defn} 
A \textit{flip} of an arc $i\in T$ is a transformation of a triangulation $T$ that removes $i$ and replaces it with a (unique) different arc $i'$ that, together with the remaining arcs, forms a new triangulation $\mu_i (T)$.  
\end{defn}

The arc $i'$ can be obtained directly by applying the following steps.
\begin{enumerate}[1.]
\item Identify the quadrilateral in $T$ that has $i$ as its diagonal. The sides of this quadrilateral may be formed by other arcs in $T$ as well as by the boundary segments of $S$. 
\item Replace the arc $i$ by the unique other diagonal $i'$ of this quadrilateral. 
\end{enumerate}

As the notation suggests, the flip of an arc in a triangulation $T$ corresponds to the mutation of the quiver $Q_T$ at the corresponding vertex.  Consider the following result which also holds in a more general case.   

\begin{thm}\cite{FST}
Given a triangulation $T$ of $(S,M)$ and an arc $i$ in $T$ corresponding to the vertex $i$ in the quiver $Q_T$, then  
$$Q_{\mu_i (T)} = \mu_i (Q_T).$$
\end{thm}

There is also a way to determine whether a particular arc corresponds to a red or a green vertex after a series of flips.  Note that in order to determine the coloring of vertices we introduced framed quivers.  In the geometrical setting there is a notion of \emph{laminations} that keeps track of this additional data.  See \cite{FT} for further details.   






Consider an automorphism $\tau$ on the set of arcs in the surface $S$ defined in the following way.  Let $\gamma$ be an arc in $S$, then $\tau \gamma$ is a new arc in $S$ obtained from $\gamma$ by rotating $\gamma$ clockwise until it coincides with the next arc.   If $T$ is a set of arcs in $S$, then by $\tau T$ we denote a new set of arcs obtained by applying $\tau$ to every arc in $T$.   Observe that in a similar manner we can define the inverse automorphism $\tau^{-1}$ such that $\tau \tau^{-1}$ and $\tau^{-1}\tau$ are the identity mappings on the set of arcs.

Although we can perform MGS's on triangulations using flips, when presenting proofs it is often more convenient to perform MGS's on the corresponding quiver. However, we have the following interesting property when dealing with triangulations, which will be a key idea in showing that a given MGS is of minimal length.  For general surfaces, this result is shown in \cite{BQ} using Ginzburg dg algebras.  For the convenience of the reader, we provide a different proof in the case of polygons using $\tau$-tilting theory.    

\begin{prop}\label{3.3}
Suppose $T$ is a triangulation of a polygon $S$. Construct $Q_T$ and proceed with a maximal green sequence to obtain $\overset{\smvee}{Q}_T$.  Then the resulting triangulation corresponding to $\overset{\smvee}{Q}_T$ is precisely $\tau T$.  
\end{prop}

\begin{proof}
It is known that maximal green sequences are in bijection with maximal chains in the poset of support $\tau$-tilting modules of the corresponding 2-Calabi-Yau tilted algebra $B$.  For references see both \cite{AIR} and \cite{BY}.  For quivers coming from triangulations of polygons $B$ is actually a cluster-tilted algebra.  Under this bijection the quiver $\hat{Q}$ corresponds to the algebra $B$ considered as a $B$-module, and $\overset{\smvee}{Q}$ corresponds to the zero module.  More precisely $\overset{\smvee}{Q}$ corresponds to $\tau_B B$, where $\tau_B$ is the Auslander-Reiten translation in the module category of $B$.  

Translating this bijection into the geometric setting, a support $\tau$-tilting module $M$ corresponds to a triangulation $T_M$ of the polygon.  For a description of the module category of $B$ via arcs in a polygon see \cite{CCS}.  Moreover, it is known that $\tau T_M = T_{\tau_B M}$, where $\tau$ is the automorphism on the set of arcs as defined above.  This implies that $\overset{\smvee}{Q}_T$ corresponds to the triangulation $\tau T$.  
\end{proof}

We say two arcs in a triangulation $T$ are \emph{connected} if they belong to the same triangle $\Delta$ in $T$.  In other words the corresponding vertices in the quiver are connected by an arrow.  We now introduce three basic configurations of arcs found in triangulations of polygons: fans, zigzags, and interior triangles.  Refer to Figure 1 for illustrations.  

An \textit{interior triangle} is a collection of three arcs that form a triangle. In $Q_T$, this corresponds to a 3-cycle.  An \emph{interior triangle configuration} is a maximal collection of connected interior arcs that form a closed (triangulated) polygon $\mathfrak{p}$ inscribed in $S$.   In $Q_T$, the polygon $\mathfrak{p}$ corresponds to a maximal full connected subquiver generated by arrows lying in 3-cycles which is called a \emph{3-cycle configuration}.    

A \textit{fan} is a maximal collection of one or more connected arcs that share a single vertex such that no arc in a fan is also a part of an interior triangle.  In $Q_T$, this corresponds to a single vertex or a series of vertices with arrows pointing in the same direction.

A \textit{zigzag} is a maximal collection of three or more connected arcs composed of fans, where each fan is made up of exactly two arcs.   In $Q_T$, this corresponds to a series of vertices with arrows in between that alternate in direction.

\begin{figure}[ht]
\includegraphics[scale=0.5]{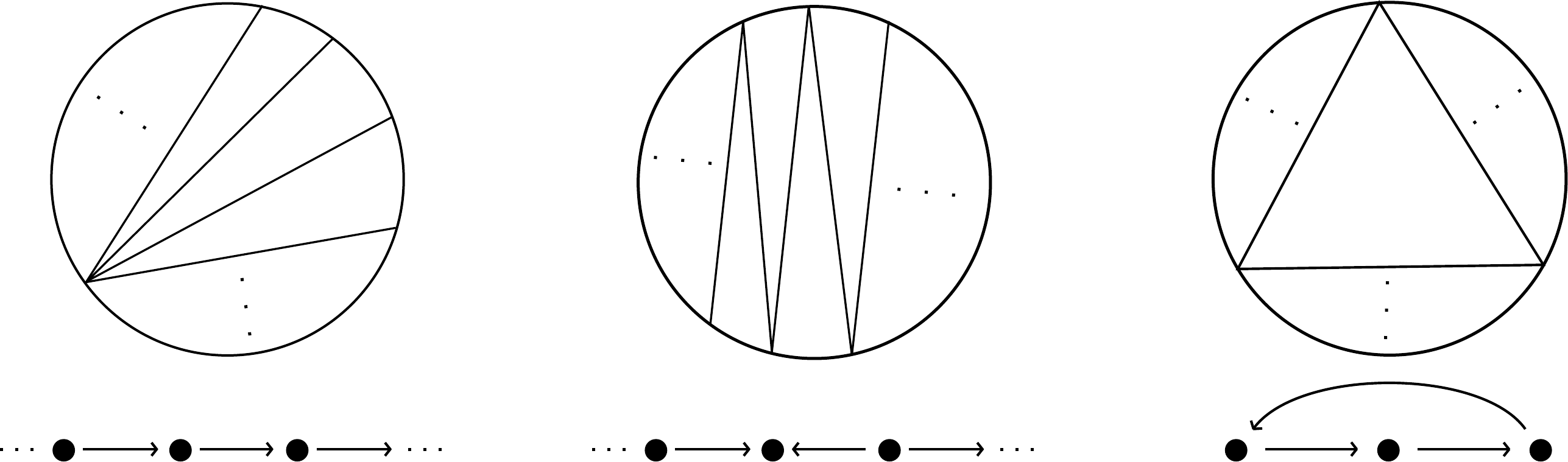}
\caption{Arcs arranged in a fan (left), a zigzag (middle), and an interior triangle (right) and their corresponding quivers.}
\end{figure}

Observe that in a general triangulation these configurations can overlap in at most two connected arcs.  Moreover, it is evident that any given triangulation is composed of some combination of these three configurations. In this paper, we develop separate procedures to find minimal length MGS's for a quiver coming from each configuration before combining them into an overarching procedure that covers all possible triangulations.

\section{Fans and Zigzags}
We begin by establishing a labeling system for vertices belonging to fans and zigzags.  Then we proceed to construct specific maximal green sequences of shortest length for these configurations.  Observe that such quivers are acyclic, hence by Lemma \ref{2.9} and Remark \ref{2.10} we know that the length of the desired sequences coincides with the number of vertices in the quiver.   

\begin{notation} \label{4.4}
For a quiver consisting of a single fan $F_i$ with $n$ vertices, label the only source $F_{i_{1}}$.  Let $F_{i_{2}}$be the vertex adjacent to $F_{i_{1}}$, and $F_{i_{3}}$ be the vertex adjacent to $F_{i_{2}}$. Continue in this way until the last vertex labeled $F_{i_{n}}$ is reached. 
\end{notation}

\begin{prop}\label{4.1} (fan procedure) 
Given an acyclic type $\mathbb{A}$ quiver $Q$ consisting of a \textit{fan} $F_{i}$ with $n$ vertices, a maximal green sequence for ${Q}$ of minimal length is given by the mutation sequence 

\centerline{$\mu_{F_{i}} \coloneqq \mu_{F_{i_{n}}}\mu_{F_{i_{n-1}}}...\mu_{F_{i_{2}}}\mu_{F_{i_{1}}}$.} 
\end{prop}

\begin{proof} 
As mentioned earlier, Lemma \ref{2.9} shows that a minimal length MGS for any acyclic quiver $Q$ is given by an admissible sequence of sources in $Q$. Therefore, it suffices to show that the mutation sequence outlined above mutates only at sources in each $Q^{i}$, where $Q^{i} \coloneqq \mu_{i}\mu_{i-1}...\mu_{1}(\hat{Q})$, and that each vertex in $Q$ is mutated exactly once. 

By construction, $F_{i_1}$ is a source, so thus far the procedure operates according to the desired assumptions. When it is mutated, its arrow going to $F_{i_2}$ changes direction, and now goes to $F_{i_1}$, making $F_{i_2}$ a source. Analogously, $F_{i_3}$ will become a source in $Q^2$.  It is easy to see that this process will continue all the way through $F_{i_n}$.  Hence, the mutation sequence gives an admissible sequence of sources in $Q$. 
\end{proof}

\begin{notation}\label{4.3}
 For a quiver consisting of a single zigzag $Z$ label each source $C_i$ and each sink $K_{j}$ for some $i$ and $j$.   
 \end{notation}

\begin{notation}
 For a quiver consisting of a single zigzag $Z$, choose an ordering of all sources $C_{i}$ and an ordering of all sinks $K_{j}$.  Let $\mu_{z_{{c}}}$  and $\mu_{z_{{k}}}$ denote the mutation sequences corresponding to this ordering. 
 \end{notation}

In the construction outlined below it will be shown that the order of mutations within each set $Z_{c}$ and $Z_{k}$ is irrelevant.

\begin{prop}\label{4.2} (zigzag procedure)
Given a quiver $Q$ consisting of a single \textit{zigzag} $Z$ a maximal green sequence for $Q$ of minimal length is given by the mutation sequence\\
\centerline{$\mu_{z_{{k}}}\mu_{z_{{c}}}$.} 
\end{prop}

\begin{proof}
Using the same argument as in the beginning of Proposition \ref{4.1}, it suffices to show that the mutation sequence outlined above mutates only at sources at each step, and that each vertex in $Q$ is mutated exactly once. 

It is clear that in $Q$ each source and each sink is only connected to other sinks and other sources respectively.  When one source $C_{{i}}$ is mutated, it cannot affect any other source in $Q$ because no arrow leaves $C_{{i}}$ in $Q$. Therefore, the sequence $\mu_{z_{{c}}}$ mutates at all sources in $Q$. We claim that all $K_{{i}}$ are now sources in $\hat{Q}' \coloneqq \mu_{z_{{c}}}(\hat{Q})$. The arrows incident to each $K_{{i}}$ in $\hat{Q}$ were all going into $K_{{i}}$. Since each arrow is incident to a unique source, each arrow in ${Q}$ has switched directions exactly once after $\mu_{z_{{c}}}$. It follows that each $K_{{i}}$ is a source in $\hat{Q}'$. Moreover, mutating in each $K_{{i}}$ does not affect any other $K_{{j}}$ using the same reasoning as above. It follows that the proposed mutation sequence gives an admissible sequence of sources in $Q$.  
\end{proof}

Thus, we have developed procedures for quivers consisting only of fans or zigzags. Next, we combine the two procedures to produce an MGS for acyclic quivers of type $\mathbb{A}$. 

Given a quiver $Q$ consisting of fans and zigzags, there are vertices that are part of both configurations.  In particular a fan and a zigzag can intersect in two vertices connected by an arrow, and a fan can intersect another fan in at most one vertex.  Therefore, we adapt the following notation.

\begin{notation}\label{4.10}
Consider a full subquiver $F_i$ of $Q$ composed of a fan with at least three vertices.  Let $V_c$ and $V_k$ be the unique source and the unique sink in $F_i$ respectively.    Label the subquiver $F_i\setminus \{V_c, V_k\}$, in accordance to the $i$th fan labeling as described in Notation \ref{4.4}.  Repeat this for every such fan configuration.  The remaining vertices will necessarily be sinks or sources.  Label these vertices $C_i$ and $K_j$ for some $i$ and $j$.  An example is given below.

$$\xymatrix{\cdots \ar[r] &  K_e & F_{i_2} \ar[l] & F_{i_1} \ar[l] & C_a \ar[l]\ar[r]& K_d & C_b \ar[l]\ar[r]& K_g  & F_{j_m}  \ar[l]& F_{j_{m-1}} \ar[l]& \ar[l] \cdots }$$

\end{notation}

\begin{remark}
The subscript for the distinct fans $F_i$ in $Q$ is arbitrary. This is motivated by the fact that in the mutation sequences outlined below no two fans will ever be connected. Moreover, we treat various zigzags in $Q$ as a collection of alternating sinks $C_i$ and sources $K_j$, and the subscript used for labeling these vertices is also arbitrary.  This is motivated by the fact that no two $C_{i}, C_{j}$ and no two $K_{i}, K_{j}$ will ever be connected during the mutation sequences outlined below. 
\end{remark}

\begin{prop} \label{4.5} (zigzag-fan procedure)
Given a quiver $Q$ consisting of fans and zigzags, a maximal green sequence for $Q$ of minimal length is given by the mutation sequence $\mu_{z_k}\mu_f\mu_{z_c}$ made up of three distinct parts:\\

\hspace{3cm} \textit{Source sequence:} \hspace{1.5mm}$\mu_{z_c} \coloneqq  \mu_{C_{i_{g}}}\mu_{C_{i_{g-1}}}\cdots \; \mu_{C_{i_{2}}}\mu_{C_{i_{1}}}.$ 

\hspace{3cm} \textit{Fan sequence:} \hspace{7.0mm}$\mu_{f} \coloneqq  \mu_{F_{h}}\mu_{F_{h-1}}\cdots \;\mu_{F_{2}}\mu_{F_{1}}.$ 

\hspace{3cm} \textit{Sink sequence:} \hspace{5.0mm}$\mu_{z_k} \coloneqq  \mu_{K_{j_{k}}}\mu_{K_{j_{k-1}}}\cdots \;\mu_{K_{j_{2}}}\mu_{K_{j_{1}}}.$ \\


\noindent Here, there are $g$ sources, $h$ fans, and $k$ sinks present in $Q$, and each sequence above mutates in the corresponding set of vertices.   
\end{prop}

\begin{proof}
We begin by considering a local configuration in $\hat{Q}$ consisting of a fan $F_{i}$ containing $r+2$ vertices, where $r\geq 1$, together with two vertices $V_{1}$ and $V_{2}$ attached to the endpoints of the fan.  Label $\hat{Q}$ according to Notation \ref{4.10} and obtain the following diagram.  Here we omit the frozen vertices.   
$$\xymatrix{\cdots  V_1 & C_a \ar[r]\ar[l] & F_{i_1}\ar[r] & F_{i_2}\ar[r] & \cdots \ar[r] &F_{i_{r-1}}\ar[r]  &F_{i_r} \ar[r] & K_b & V_2\ar[l]  \cdots }$$

Observe that in $\mu_{z_c}(\hat{Q})$ the vertex $F_{i_1}$ is a source.  Next we perform the fan procedure $\mu_f$, which includes the mutation sequence $\mu_{F_i}$.  By Proposition \ref{4.1} this sequence makes the vertices $F_{i_j}$ for all $j = 1, \dots , r$ red.  Moreover, it is easy to see that in $\mu_f\mu_{z_c}(\hat{Q})$ there is an arrow $K_b \rightarrow F_{i_r}$.  

Next, we claim that there is also an arrow $K_b\rightarrow V_2$ in $\mu_f\mu_{z_c}(\hat{Q})$, thus making $K_b$ a source.  If $V_2$ was originally a source in $Q$, then it was mutated exactly once during $\mu_f\mu_{z_c}$ yielding an arrow $K_b\rightarrow V_2$ in the resulting quiver.  If $V_2$ was not a source in $Q$, then it was a part of some fan $F_j$.  However, during the fan procedure applied to $\mu_{z_c}(\hat{Q})$ there was a source mutation at $V_2$, again producing an arrow $K_b \rightarrow V_2$ in $\mu_f\mu_{z_c}(\hat{Q})$ as desired.  This shows that $K_b$ is indeed a source in this quiver.  Therefore, applying the final set of mutations $\mu_{z_k}$ results in a source mutation at $K_b$.  Observe, that a similar argument implies that $V_1$ will be mutated once along $\mu_{z_k}\mu_f\mu_{z_c}$, and this will be a source mutation.  Hence, we showed that the mutation sequence $\mu_{z_k}\mu_f\mu_{z_c}$ makes every vertex, in the diagram depicted above, red.  This implies that every such configuration of a fan together with the two vertices on the ends becomes red after performing the mutation sequence described in the proposition.  Observe that if $V_1$ or $V_2$ were not present, then the same conclusion still holds.  

For a zigzag configuration $Z$ in $\hat{Q}$ the mutation sequence $\mu_{z_k}\mu_f\mu_{z_c}$  reduces to $\mu_{z_k}\mu_{z_c}$, which by Proposition \ref{4.2} results in all red vertices in $Z$.  It follows that since the procedure produces a maximal green sequence for fan and zigzag configurations, any given $\hat{Q}$ can be thought of as a composition of several such components.  Observe, that mutations in the disjoint components do not impact each other, and by above, the mutations in the overlapping components are also well-behaved, meaning all mutations happen at sources.  Therefore, the mutation sequence given in the proposition dictates to mutate in all vertices exactly once and every mutation is a source mutation. 
\end{proof}

\begin{ex}
Consider the following acyclic quiver $Q$.  Bellow we apply the labeling described in Notation \ref{4.10}.  

$$\xymatrix@C=22pt{1 \ar[r] & 2 & 3 \ar[r]\ar[l] & 4 & 5 \ar[l] & 6 \ar[l] & 7 \ar[l] & 8 \ar[l]\ar[r] & 9 \ar[r]& 10 \ar[r] & 11 }$$

$$\xymatrix@C = 14pt{C_1\ar[r] & K_1 & C_2 \ar[r]\ar[l] & K_2 & F_{1_3}\ar[l] & F_{1_2}\ar[l] & F_{1_1}\ar[l] & C_3\ar[l]\ar[r] & F_{2_1}\ar[r]& F_{2_2}\ar[r] & K_3 }$$

\noindent By Proposition \ref{4.5} a maximal green sequence for $Q$, separated into three distinct parts, is shown below.

$$\mu_{11}\mu_4\mu_2\;\;\mu_{10}\mu_{9}\mu_5\mu_6\mu_7 \;\;\mu_8 \mu_3\mu_1$$

\end{ex}

\section{Interior Triangles}

In this section we consider quivers that are composed only of 3-cycles.  Here we develop a procedure that yields a minimal length MGS for such quivers.  Observe, that unlike the acyclic quivers we discussed in the previous section, these quivers do not admit an admissible sequence of sources.  Therefore, the situation becomes more complicated.  We begin by looking at a single 3-cycle.  

\begin{ex}\label{5.1}
Consider the following quiver with 3 vertices.  A minimal length of a maximal green sequence is 4.  Such a sequence can be obtained by mutating at an arbitrary vertex and moving clockwise or counterclockwise along the cycle until the initial vertex is mutated again.  Below we see that $\mu_1\mu_3\mu_2\mu_1$ is an MGS for this quiver.  We omit the frozen vertices, but the ones emphasized in bold correspond to green vertices.  

$$\xymatrix@C=10pt @R=10pt{& {\bf 3} \ar[dl]&&&&&{\bf 3} &&&&& {\bf 3} &&&&& 3\ar[dl] &&&&& 3\ar[dr] \\ 
{\bf 1} \ar[rr]&& {\bf 2} \ar[ul]&\ar[r]^{\mu_1}&& 1 \ar[ur]&&{\bf 2} \ar[ll] &\ar[r]^{\mu_2}&& 1\ar[rr]\ar[ur] &&2 &\ar[r]^{\mu_3}&& {\bf1}\ar[rr] &&2  &\ar[r]^{\mu_1}&& 1\ar[ur] &&2\ar[ll] }$$
\end{ex}

Based on the sequence for a single 3-cycle we build a maximal green sequence for configurations of adjoined 3-cycles.  Throughout this section we take $Q$ to be a quiver consisting entirely of joined 3-cycles.  We begin by investigating the structure of these cycles.    

\begin{defn}
A \emph{shared vertex} is a vertex that belongs to exactly two distinct 3-cycles.  A \emph{leader} is a non-shared vertex $l$ such that there exists an arrow from $l$ to a shared vertex.  Finally, a {\emph {follower}} is a non-shared vertex $f$ such that there exists an arrow from $f$ to a leader.  
\end{defn}

Note that a given vertex of $Q$ can be a part of at most two 3-cycles. This follows from the fact that quivers arising from triangulations of polygons can have at most two arrows going in and at most two arrows going out at each vertex. 

\begin{defn}
Let $Q$ be a type $\mathbb{A}$ quiver consisting entirely of $3$-cycles. Define the corresponding graph $G_{Q}$ where $ \{ i \mid T_{i} \text{\; is a 3-cycle in\;} Q\}$ is the set of vertices of $G_Q$ and $\{ \xymatrix@C=10pt{i \ar@{-}[r]& j} \mid T_i \text{\;and\;} T_j \text{\; share a vertex}\}$ is the set of edges.    
\end{defn}

\begin{lem}\label{5.0}
The graph $G_Q$ is a tree.  
\end{lem}

\begin{proof}
Let $Q$ consist of $t$ 3-cycles.  It suffices to show that $G_Q$ cannot be a cycle.  Otherwise the quiver $Q$ can be reduced to a quiver $Q'$ by removing sufficiently many 3-cycles such that $G_{Q'}$ is a cycle.  Observe that $Q'$ would still be of type $\mathbb{A}$, because the same reduction can be realized in the corresponding triangulated polygon $(S,M)$ yielding another triangulated polygon $(S', M')$.  

Hence, suppose $G_Q$ is a cycle with $t$ vertices.  Then, if we forget about the orientation of the arrows in $Q$ we obtain the following underlying structure. Here the triangles with dotted edges correspond to an arbitrary number of 3-cycles.

$$\xymatrix@C=10pt@R=10pt{&&\bullet \ar@{-}[dr]\ar@{-}[dl]\\
&\bullet \ar@{-}[dd]\ar@{-}[dl]\ar@{-}[rr]&& \bullet \ar@{--}[dd]\ar@{--}[dr] \\
\bullet \ar@{-}[dr]&&&& \bullet \ar@{--}[dl]\\
&\bullet \ar@{--}[rr]\ar@{--}[dr]&& \bullet \\
&&\bullet \ar@{--}[ur] }$$

Observe that if $Q$ has the configuration above than it consists of $2t$ vertices.  Therefore, by Remark \ref{3.0} the quiver $Q$ can be realized as a triangulation $T$ of an $(2t+3)$-gon.  Moreover, since each arc of $T$ belongs to an interior triangle, then $T$ is an interior triangle configuration.  By definition, the arcs of $T$ form a maximal closed polygon $\mathfrak{p}$ inscribed in the $(2t+3)$-gon.  Thus, $\mathfrak{p}$ consists of $\frac{2t+3}{2}$ vertices.  This implies that $2t+3$ is even, which is a contradiction. This shows that $G_Q$ is a tree.  

\end{proof}

Next we introduce the notion of regions in $Q$.  These definition are motivated by the geometric approach of surfaces and triangulations.  

\begin{defn}\label{5.3}
Let $R_{m}$ be a full subquiver of $Q$ consisting of a single $3$-cycle containing both a leader and a follower. Call it the \textit{innermost region of Q}. Call $T_{m}$ the only $3$-cycle in $R_{m}$.
 \end{defn}
 
\begin{remark} 
Observe that by Lemma \ref{5.0} the innermost region always exists, and in general, there are many possibilities for $R_{m}$. 
\end{remark}

\begin{defn} \label{5.5}
For the same $Q$, define the full subquiver $R_{1}$ of $Q$ to be the union of all $3$-cycles containing leaders and followers in $Q$ but not in $R_{m}$. Call it the \textit{outermost region of Q}. Denote the cycles in $R_{1}$ as $T_{1_{1}}, T_{1_{2}}, ..., T_{1_{n}}$, where the ordering of the second index is arbitrary. For a cycle $T_{1_{i}}$, label the leader $L_{1_{i}}$ and the follower $F_{1_{i}}$. 
\end{defn}

\begin{defn} \label{5.6}
Define the full subquiver $R_{2}$ to be the union of the cycles $T_{2_{1}}, T_{2_{2}}, ..., T_{2_{m}}$ with leaders and followers in $Q_{2} \coloneqq Q \setminus \{V \in R_{1} \mid V$ is a leader or follower in $R_{1} \}$ with the added restriction that $T_{m} \cap R_{2}=\emptyset$. In other words, they are the outermost region of $Q_{2}$. For a cycle $T_{2_{i}}$, label the leader $L_{2_{i}}$ and the follower $F_{2_{i}}$.  Call $R_{2}$ the \textit{region 2} of $Q$. 
\end{defn}

\begin{defn} \label{5.7}
Define the full subquiver $R_{i}$ to be the union of cycles $T_{i_{1}}, T_{i_{2}}, ..., T_{i_{r}}$ with leaders and followers in $Q_{i} \coloneqq Q_{i-1} \setminus \{V \in R_{i-1} \mid V$ is a leader or follower in $R_{i-1} \}$ with the added restriction that $T_{m} \cap R_{i}= \emptyset$. For a cycle $T_{i_{j}}$, label the leader $L_{i_{j}}$ and the follower $F_{i_{j}}$. Call $R_{i}$ the  \textit{region $i$} of $Q$.

There can be at most one $3$-cycle that shares a vertex with $T_{m}$. This will be the region $R_{m'}$, where $m' \geq i$ for all $R_{i} \in Q$. Call $T_{m'_{1}}$ the only $3$-cycle in $R_{m'}$. 
\end{defn}

\begin{ex}
Consider the following quiver whose vertices are labeled according to the definitions above.  Here the innermost region consists of vertices $V_{m_1}$, $V_{m_2}$, and $S_{6_1}$.  \\

\begin{tikzpicture}[
            > = stealth, 
            shorten > = 1pt, 
            auto,
            node distance = 1.5cm, 
            semithick 
        ]

        \tikzstyle{every state}=[
            draw = black,
            thick,
            fill = white,
            minimum size = 3mm
        ]

        \node[draw=none, fill=none] (1) {$S_{6_{1}}$};
        \node[draw=none, fill=none] (2) [above right of=1] {$F_{6_{1}}$};
        \node[draw=none, fill=none] (3) [right of=1] {$L_{6_{1}}$};
        \node[draw=none, fill=none] (4) [left of=1] {$V_{m_{2}}$};
        \node[draw=none, fill=none] (5) [above left of=1] {$V_{m_{1}}$};
        \node[draw=none, fill=none] (6) [above left of=2] {$L_{5_{1}}$};
        \node[draw=none, fill=none] (7) [right of=3] {$F_{4_{2}}$};
        \node[draw=none, fill=none] (8) [above right of=3] {$L_{4_{2}}$};
        \node[draw=none, fill=none] (9) [above right of=2] {$F_{5_{1}}$};
        \node[draw=none, fill=none] (10) [above left=7mm of 6] {$L_{4_{1}}$};     
        \node[draw=none, fill=none] (11) [left of=6] {$F_{4_{1}}$};     
        \node[draw=none, fill=none] (12) [above right of=9] {$L_{1_{2}}$};     
        \node[draw=none, fill=none] (13) [right of=9] {$F_{1_{2}}$};     
        \node[draw=none, fill=none] (14) [right of=8] {$F_{3_{2}}$};
        \node[draw=none, fill=none] (15) [right=5mm of 14] { };
        \node[draw=none, fill=none] (16) [left=4mm of 4] { };         
\node[draw=none, fill=none] (17) [left=4mm of 5] { };       
                  \node[draw=none, fill=none] (18) [left of=11] {$F_{1_{6}}$};                    
                  \node[draw=none, fill=none] (19) [left of=10] {$F_{3_{1}}$};  
                  \node[draw=none, fill=none] (20) [right of=7] {$L_{3_{2}}$};
                                  \node[draw=none, fill=none] (23) [right of=20] {$F_{1_{5}}$};                                    
                       \node[draw=none, fill=none] (24) [right of=14] {$L_{2_{2}}$};      
                       \node[draw=none, fill=none] (25) [left=6cm of 1] { };      
                       \node[draw=none, fill=none] (26) [below right of=18] {$L_{1_{6}}$};
                       \node[draw=none, fill=none] (27) [above right of=19] {$L_{3_{1}}$};
                            \node[draw=none, fill=none] (28) [right of=27] {$F_{2_{1}}$};
                                 \node[draw=none, fill=none] (29) [above right of=27] {$L_{2_{1}}$};
                                         \node[draw=none, fill=none] (30) [below left of=23] {$L_{1_{5}}$};
     \node[draw=none, fill=none] (31) [above right of=14] {$F_{2_{2}}$};
     \node[draw=none, fill=none] (32) [right of=24] {$F_{1_{4}}$};
     \node[draw=none, fill=none] (33) [above right of=24] {$L_{1_{4}}$};
     \node[draw=none, fill=none] (34) [above right of=31] {$L_{1_{3}}$};
     \node[draw=none, fill=none] (35) [left of=34] {$F_{1_{3}}$};
     \node[draw=none, fill=none] (36) [right of=28] {$L_{1_{1}}$};
          \node[draw=none, fill=none] (37) [above right of=28] {$F_{1_{1}}$};
                                              
             \path[->] (1) edge node { } (2);
        \path[->] (2) edge node { } (3);
               \path[->] (3) edge node { } (1);
        \path[->] (5) edge node { } (4);
        \path[->] (4) edge node { } (1);
        \path[->] (1) edge node { } (5);
       \path[->] (7) edge node { } (8);
        \path[->] (8) edge node { } (3);
        \path[->] (6) edge node { } (2);
       \path[->] (2) edge node { } (9);
        \path[->] (9) edge node { } (6);
        \path[->] (20) edge node { } (23);
          \path[<-] (14) edge node { } (24);
               \path[<-] (11) edge node { } (6);
               \path[->] (14) edge node { } (20);
               \path[->] (7) edge node { } (14);
               \path[<-] (37) edge node { } (28);
               \path[<-] (28) edge node { } (36);
               \path[<-] (36) edge node { } (37);
     \path[->] (14) edge node { } (31);
      \path[->] (24) edge node { } (32);
     \path[->] (33) edge node { } (24);
      \path[->] (31) edge node { } (24);
      \path[->] (32) edge node { } (33);
      \path[->] (23) edge node { } (30);
      \path[->] (30) edge node { } (20);
      \path[->] (35) edge node { } (34);
      \path[->] (34) edge node { } (31);
      \path[->] (31) edge node { } (35);
      \path[->] (18) edge node { } (26);
      \path[->] (26) edge node { } (11);
      \path[->] (10) edge node { } (19);
      \path[->] (19) edge node { } (27);
      \path[->] (27) edge node { } (10);
      \path[->] (27) edge node { } (28);
      \path[->] (28) edge node { } (29);
      \path[->] (29) edge node { } (27);
      \path[<-] (6) edge node { } (10);
      \path[->] (20) edge node { } (7);
        \path[->] (11) edge node { } (10);
        \path[->] (3) edge node { } (7);
      \path[<-] (12) edge node { } (13);
      \path[->] (9) edge node { } (13);
    =                         \path[->] (12) edge node { } (9);
             
                  \path[->] (11) edge node { } (18);
       \end{tikzpicture}

\end{ex}

\begin{lem}\label{5.8}
Every region $R_{i}$ is a disjoint union of $3$-cycles. 
\end{lem}

\begin{proof}
Suppose there are two $3$-cycles $T_{a}$ and $T_{b}$ attached in $Q$, but are not a part of $R_{m}$ or $R_{m'}$. In order for them to be in the same region, it is clear that they must be the same number of $3$-cycles away from $R_{m}$ (and $R_{m'}$). In other words, there must be a path from ${a}$ to ${m'}$ in $G_Q$ that does not pass through ${b}$, and the same must be true for a path from ${b}$ to ${m'}$. In order for this to be true, $G_Q$ must contain a cycle, which by Lemma \ref{5.0} is impossible.  Therefore, $T_a$ and $T_b$ must belong to different regions of $Q$.    
\end{proof}

Next we describe a procedure that yields MGS's for such quivers.  It begins by mutating in the outermost regions first, then moving from one region to the next performing similar mutation sequences until the innermost region is reached.  Afterwards it dictates to move back through the regions in the decreasing order while following the same pattern of mutations in each region.  

\begin{thm}\label{5.9}(3-cycle procedure)
 The following procedure yields an MGS for quivers $Q$ consisting entirely of connected 3-cycles.
 \begin{enumerate}[1.]
\item  Establish $R_{1}, R_{2}, ..., R_{m}$ as outlined in Definitions \ref{5.3} and \ref{5.5}-\ref{5.7}. Label the vertices of $R_{m}$ as $V_{m_{1}}, V_{m_{2}}, S_{m'_{1}}$, where $S_{m'_{1}} \in R_{m'}$. Now consider $\hat{Q}$ from this point on. 
\item  Mutate in all $L_{1_{i}}$. Call this mutation sequence $\mu_{L_{1}}$. 
\item Move consecutively from region $R_i$ to region $R_{i+1}$ and mutate at all leaders $L_{i_j}$ for every region $R_{i}$, where $1 \leq i \leq m'$. Let $\mu_{L_{i}}$ be the mutation sequence for $R_{i}$. 
\item Mutate the vertices of $R_{m}$ starting with an arbitrary vertex and then move in a cyclic order around $R_{m}$, until the vertex that was first mutated is mutated again. Call this mutation sequence $\underline{\mu_{m}}$. 
\item Mutate at $F_{m'_{1}}$ and then at $L_{m'_{1}}$. Call this mutation sequence $\underline{\mu_{m'_{1}}}$. 
\item Move consecutively from region $R_i$ to region $R_{i-1}$, where $1\leq i <m'$,  and mutate at $F_{i_j}$ and then at $L_{i_j}$ for every cycle $T_{i_{j}}\in R_i$.  Let $\underline{\mu_{i}}$ be the mutation sequence for $R_i$. 
\end{enumerate}
The resulting maximal green sequence for $Q$ is
$$\mu_Q : = \; \underline{\mu_1}\;\underline{\mu_2}\;\cdots\; \underline{\mu_{m'-1}}\;\underline{\mu_{m'}}\;\underline{\mu_m}\;\mu_{L_{m'}}\;\mu_{L_{m'-1}}\;\cdots \;\mu_{L_2}\;\mu_{L_1}.$$

\end{thm}

\begin{remark} \label{5.22}
Labeling of cycles within a region is arbitrary because it will be shown that mutations in different cycles belonging to the same region commute.  
\end{remark}

\begin{proof}
Observe that if $Q$ consists of a single 3-cylcle then $T_m = Q$ and by Example \ref{5.1} the theorem holds.  

We induct on the number of $3$-cycles. Consider what happens when $Q$ consists of two cycles, as shown in the diagram below. \\ \\
\begin{tikzpicture}[
            > = stealth, 
            shorten > = 1pt, 
            auto,
            node distance = 2.5cm, 
            semithick 
        ]

        \tikzstyle{every state}=[
            draw = black,
            thick,
            fill = white,
            minimum size = 1mm
        ]
      
    \node[draw=none, fill=none] (1) {$V_{m_{1}}$};
    \node[draw=none, fill=none] (2) [above right of=1] {$V_{m_{2}}$};
  \node[draw=none, fill=none] (3) [right of=1] {$S_{1_{1}}$};
  \node[draw=none, fill=none] (4) [above right of=3] {$L_{1_{1}}$};
  \node[draw=none, fill=none] (5) [right of=3] {$F_{1_{1}}$};
    \node[draw=none, fill=none] (6) [left=3.8cm of 1] { };  
    
    \path[->] (2) edge node { } (1);
     \path[->] (1) edge node { } (3);
     \path[->] (3) edge node { } (2);
     \path[->] (4) edge node { } (3);
     \path[->] (3) edge node { } (5);
     \path[->] (5) edge node { } (4);

  \end{tikzpicture} \\ 
We now consider $\hat{Q}$ and omit the frozen vertices in our diagrams.  According to the procedure, we begin by mutating $\hat{Q}$ at $L_{1_{1}}$, shown below on the left.  Note that in terms of frozen vertices, there exist two arrows starting at $F_{1_{1}}$, in particular $F_{1_1}\rightarrow F_{1_1}'$ and $F_{1_1}\rightarrow L_{1_1}'$.  Moreover, we have an arrow $L_{1_{1}}' \rightarrow L_{1_{1}}$. 

Next we can mutate at an arbitrary vertex in $T_{m}$. In this proof we chose to start with $V_{m_{2}}$, and the fact that the procedure still holds if another vertex is chosen is left for the reader to check.  Note that in $\mu_{V_{m_{2}}}\mu_{L_{1_{1}}}(\hat{Q})$ the vertex $S_{1_{1}}$ is green, with arrows $S_{1_1}\rightarrow S_{1_1}'$ and $S_{1_1}\rightarrow V_{m_{2}}'$, and $V_{m_{2}}$ is red, with an arrow $V_{m_{2}}'\rightarrow V_{m_2}$. The diagram below on the right depicts $\mu_{V_{m_{2}}}\mu_{L_{1_{1}}}(\hat{Q})$. \\ \\
\begin{minipage}{.3\textwidth}
\begin{tikzpicture}[
            > = stealth, 
            shorten > = 1pt, 
            auto,
            node distance = 2.5cm, 
            semithick 
        ]

        \tikzstyle{every state}=[
            draw = black,
            thick,
            fill = white,
            minimum size = 1mm
        ]
      
    \node[draw=none, fill=none] (1) { ${ V}_{m_{1}}$};
    \node[draw=none, fill=none] (2) [above right of=1] {$V_{m_{2}}$};
  \node[draw=none, fill=none] (3) [right of=1] { $S_{1_{1}}$};
  \node[draw=none, fill=none] (4) [above right of=3] {$L_{1_{1}}$};
  \node[draw=none, fill=none] (5) [right of=3] { $F_{1_{1}}$};
  
     \path[->] (2) edge node { } (1);
     \path[->] (1) edge node { } (3);
     \path[->] (3) edge node { } (2);
     \path[<-] (4) edge node { } (3);
          \path[<-] (5) edge node { } (4);
\end{tikzpicture}
\end{minipage}%
                      \begin{minipage}{0.3\textwidth}
\centerline{\hspace{10mm}$\xrightarrow{\makebox[2cm]{$\mu_{V_{m_{2}}}$}}$}
\end{minipage}%
 \begin{minipage}{.2\textwidth}   
\begin{tikzpicture}[
            > = stealth, 
            shorten > = 1pt, 
            auto,
            node distance = 2.5cm, 
            semithick 
        ]

        \tikzstyle{every state}=[
            draw = black,
            thick,
            fill = white,
            minimum size = 1mm
        ]
      
    \node[draw=none, fill=none] (1) {$V_{m_{1}}$};
    \node[draw=none, fill=none] (2) [above right of=1] {$V_{m_{2}}$};
  \node[draw=none, fill=none] (3) [right of=1] {$S_{1_{1}}$};
  \node[draw=none, fill=none] (4) [above right of=3] {$L_{1_{1}}$};
  \node[draw=none, fill=none] (5) [right of=3] {$F_{1_{1}}$};
      
    \path[<-] (2) edge node { } (1);
    
     \path[<-] (3) edge node { } (2);
     \path[<-] (4) edge node { } (3);
          \path[<-] (5) edge node { } (4);

  \end{tikzpicture}
  \end{minipage}%
  \\ \\
We now mutate in $V_{m_{1}}$. It is a source, so no frozen vertices except for $V_{m_1}'$ are affected. The diagram below on the left depicts $\mu_{V_{m_{1}}}\mu_{V_{m_{2}}}\mu_{L_{1_{1}}}(\hat{Q})$. Next, we mutate in $S_{1_{1}}$. Since there is an arrow $V_{m_{2}}\rightarrow S_{1_{1}}$ in $\mu_{V_{m_{1}}}\mu_{V_{m_{2}}}\mu_{L_{1_{1}}}(\hat{Q})$, the arrow $V_{m_{2}}\rightarrow V_{m_{2}}'$ will be eliminated, and $V_{m_{2}}$ will have an arrow ending in $S_{1_{1}}'$.  Hence, $V_{m_2}$ is again green in $\mu_{S_{1_{1}}}\mu_{V_{m_{1}}}\mu_{V_{m_{2}}}\mu_{L_{1_{1}}}(\hat{Q})$, which is depicted below on the right. \\ \\
  \begin{minipage}{.3\textwidth}
\begin{tikzpicture}[
            > = stealth, 
            shorten > = 1pt, 
            auto,
            node distance = 2.5cm, 
            semithick 
        ]

        \tikzstyle{every state}=[
            draw = black,
            thick,
            fill = white,
            minimum size = 1mm
        ]
      
       \node[draw=none, fill=none] (1) {$V_{m_{1}}$};
    \node[draw=none, fill=none] (2) [above right of=1] {$V_{m_{2}}$};
  \node[draw=none, fill=none] (3) [right of=1] {$S_{1_{1}}$};
  \node[draw=none, fill=none] (4) [above right of=3] {$L_{1_{1}}$};
  \node[draw=none, fill=none] (5) [right of=3] {$F_{1_{1}}$};
      
    \path[->] (2) edge node { } (1);
    
     \path[<-] (3) edge node { } (2);
     \path[<-] (4) edge node { } (3);
          \path[<-] (5) edge node { } (4);
\end{tikzpicture}
\end{minipage}%
                      \begin{minipage}{0.3\textwidth}
\centerline{\hspace{10mm}$\xrightarrow{\makebox[2cm]{$\mu_{S_{1_{1}}}$}}$}
\end{minipage}%
 \begin{minipage}{.2\textwidth}   
\begin{tikzpicture}[
            > = stealth, 
            shorten > = 1pt, 
            auto,
            node distance = 2.5cm, 
            semithick 
        ]

        \tikzstyle{every state}=[
            draw = black,
            thick,
            fill = white,
            minimum size = 1mm
        ]
      
    \node[draw=none, fill=none] (1) {$V_{m_{1}}$};
    \node[draw=none, fill=none] (2) [above right of=1] {$V_{m_{2}}$};
  \node[draw=none, fill=none] (3) [right of=1] {$S_{1_{1}}$};
  \node[draw=none, fill=none] (4) [above right of=3] {$L_{1_{1}}$};
  \node[draw=none, fill=none] (5) [right of=3] {$F_{1_{1}}$};
      
    \path[->] (2) edge node { } (1);
    
     \path[->] (3) edge node { } (2);
     \path[->] (4) edge node { } (3);
          \path[<-] (5) edge node { } (4);
          \path[->] (2) edge node { } (4);

  \end{tikzpicture}
  \end{minipage}%
  \\ \\
We now mutate in $V_{m_{2}}$. Observe that prior to the mutation, the only nonfrozen vertex with an arrow ending in $V_{m_{2}}$ is $S_{1_{1}}$. The only frozen vertex that $V_{m_{2}}$ has an arrow to is $S_{1_{1}}'$, so during the mutation $\mu_{V_{m_2}}$ the arrow $S_{1_{1}}\rightarrow S_{1_{1}}'$ will be cancelled, leaving only an arrow $V_{m_{2}}'\rightarrow S_{1_{1}}$. The diagram below on the left depicts $Q^{m} \coloneqq \mu_{V_{m_{2}}}\mu_{S_{1_{1}}}\mu_{V_{m_{1}}}\mu_{V_{m_{2}}}\mu_{L_{1_{1}}}(\hat{Q})$. The only green vertex remaining in $Q^{m}$ is $F_{1_{1}}$. Mutating in $F_{1_{1}}$ eliminates the arrow $L_{1_{1}}\rightarrow L_{1_{1}}'$ and adds an arrow $L_{1_{1}}\rightarrow F_{1_{1}}'$, turning $L_{1_1}$ green in $\mu_{F_{1_{1}}}(Q^{m})$.  The resulting quiver is shown below on the right. \\ \\
    \begin{minipage}{.3\textwidth}
\begin{tikzpicture}[
            > = stealth, 
            shorten > = 1pt, 
            auto,
            node distance = 2.5cm, 
            semithick 
        ]

        \tikzstyle{every state}=[
            draw = black,
            thick,
            fill = white,
            minimum size = 1mm
        ]
      
        \node[draw=none, fill=none] (1) {$V_{m_{1}}$};
    \node[draw=none, fill=none] (2) [above right of=1] {$V_{m_{2}}$};
  \node[draw=none, fill=none] (3) [right of=1] {$S_{1_{1}}$};
  \node[draw=none, fill=none] (4) [above right of=3] {$L_{1_{1}}$};
  \node[draw=none, fill=none] (5) [right of=3] {$F_{1_{1}}$};
      
    \path[<-] (2) edge node { } (1);
    \path[->] (3) edge node { } (1);
     \path[<-] (3) edge node { } (2);
             \path[<-] (5) edge node { } (4);
          \path[<-] (2) edge node { } (4);
  
\end{tikzpicture}
\end{minipage}%
                      \begin{minipage}{0.3\textwidth}
\centerline{\hspace{10mm}$\xrightarrow{\makebox[2cm]{$\mu_{F_{1_{1}}}$}}$}
\end{minipage}%
 \begin{minipage}{.2\textwidth}   
\begin{tikzpicture}[
            > = stealth, 
            shorten > = 1pt, 
            auto,
            node distance = 2.5cm, 
            semithick 
        ]

        \tikzstyle{every state}=[
            draw = black,
            thick,
            fill = white,
            minimum size = 1mm
        ]
      
     \node[draw=none, fill=none] (1) {$V_{m_{1}}$};
    \node[draw=none, fill=none] (2) [above right of=1] {$V_{m_{2}}$};
  \node[draw=none, fill=none] (3) [right of=1] {$S_{1_{1}}$};
  \node[draw=none, fill=none] (4) [above right of=3] {$L_{1_{1}}$};
  \node[draw=none, fill=none] (5) [right of=3] {$F_{1_{1}}$};
      
    \path[<-] (2) edge node { } (1);
      \path[->] (3) edge node { } (1);
           \path[<-] (3) edge node { } (2);
             \path[->] (5) edge node { } (4);
          \path[<-] (2) edge node { } (4);

  \end{tikzpicture}
  \end{minipage}%
  \\ \\
Finally, the only green vertex in $\mu_{F_{1_{1}}}(Q^{m})$ is $L_{1_{1}}$. Moreover, there is a unique arrow starting at a nonfrozen vertex and ending in $L_{1_1}$, and it is coming from $F_{1_{1}}$. Mutating in $L_{1_{1}}$ cancels the arrow $F_{1_{1}}\rightarrow F_{1_{1}}'$ and changes nothing else in terms of frozen vertices.  Therefore, we obtain $\mu_{L_{1_{1}}}\mu_{F_{1_{1}}}(Q^{m}) \cong \check{Q}$. 
\\
\begin{center}
\begin{tikzpicture}[
            > = stealth, 
            shorten > = 1pt, 
            auto,
            node distance = 2.5cm, 
            semithick 
        ]

        \tikzstyle{every state}=[
            draw = black,
            thick,
            fill = white,
            minimum size = 1mm
        ]
      
     \node[draw=none, fill=none] (1) {$V_{m_{1}}$};
    \node[draw=none, fill=none] (2) [above right of=1] {$V_{m_{2}}$};
  \node[draw=none, fill=none] (3) [right of=1] {$S_{1_{1}}$};
  \node[draw=none, fill=none] (4) [above right of=3] {$L_{1_{1}}$};
  \node[draw=none, fill=none] (5) [right of=3] {$F_{1_{1}}$};
      
    \path[<-] (2) edge node { } (1);
      \path[->] (3) edge node { } (1);
           \path[<-] (3) edge node { } (2);
             \path[->] (4) edge node { } (5);
          \path[<-] (4) edge node { } (2);
  \path[->] (5) edge node { } (2);
  
  \end{tikzpicture}
\\
\end{center}

Now suppose that the procedure produces an MGS for all quivers consisting entirely of $t$ $3$-cycles, where $t \geq 2$. Consider a quiver $Q$ made up of $(t+1)$ 3-cycles.  Label $Q$ according to the labeling system described in the beginning of this section.  Choose an ordering of $L_{i_j}$ and an ordering of $\mu_{L_{i_j}}\mu_{F_{i_j}}$ within each region $i$.  Then obtain the corresponding mutation sequence $\mu_Q$ as defined in the statement of the theorem.   Note that in the theorem $\mu_Q$ is written in terms of $\mu_{L_i}$ and $\underline{\mu_{i}}$ for all $i = 1, \dots , m'$.  Therefore, it is necessary to show that each of these mutation sequences is well-defined.  This means that the order of mutations of leaders within each set $L_i$ is arbitrary, and the order of mutations $\mu_{L_{i_j}}\mu_{F_{i_j}}$ within the same region $i$ is also arbitrary.   

Observe that by Lemma \ref{5.8} each region is a disjoint union of 3-cycles.  This implies that $\mu_{L_1}$ is well-defined, because the order of mutations of $L_{1_j}$ is irrelevant.   It will be shown later that $\underline{\mu_1}$ is also well-defined.  For now, let $L_{1_Q}\in L_1$ be the the last vertex to be mutated in $\mu_Q$.  Hence, by the argument above $\mu_Q$ can be written as $\mu_{L_{1_Q}}\mu_{F_{1_Q}}\mu_0\mu_{L_{1_Q}}$, where $\mu_0$ is some remaining mutation sequence for $Q$.  

Let $\mathcal{C}$ be a full connected subquiver of $Q$ obtained by removing $F_{1_Q}$ and $L_{1_Q}$ from $Q$.  It follows that $\mathcal{C}$ is made up of $t$ 3-cycles. The diagram below depicts $\mathcal{C}$ and $Q$. 

    \begin{minipage}{.24\textwidth}
    \hspace{5cm}
    \end{minipage}%
                      \begin{minipage}{0.5\textwidth}
     \begin{tikzpicture}[
> = stealth, 
            shorten > = 4pt, ]
            
\node[draw=none, fill=none] at (7.7,0.4) {$\mathcal{C}$};
\node[draw, shape=circle, fill=black, scale=0.5] at (8.5,-2) { };
\node[draw, shape=circle, fill=black, scale=0.5] at (7.3,-3.4) { };
\node[draw, shape=circle, fill=black, scale=0.5] at (9.7,-3.4) { };
\node[draw=none, fill=none] at (7, 3) { };
\node[draw=none, fill=none] at (7, -3) { };

\draw [black] plot [smooth cycle] coordinates {(4.5,0) (6.5,2) (10.5,2) (10.5,0) (8.5,-2)};
\path [->] (8.5,-2) edge node { } (7.3,-3.4); 
\draw [->] (7.3,-3.4) -- node { } (9.7,-3.4);
\draw [->] (9.7,-3.4) -- node { } (8.5,-2);

\draw [decorate,decoration={brace,amplitude=10pt,mirror,raise=4pt},yshift=0pt]
(11.5,-3.8) -- (11.5,2.4) node [black,midway,xshift=0.8cm] {\footnotesize
$Q$};
\end{tikzpicture} 
\end{minipage}%
\\
 
Observe that $\mu_0$ is a mutation sequence for $\mathcal{C}$.  Moreover, there exists a labeling of vertices in $\mathcal{C}$ (if the quiver $\mathcal{C}$ is considered on its own) such that $\mu_0 = \mu_{\mathcal{C}}$, where $\mu_{\mathcal{C}}$ is an MGS for $\mathcal{C}$ obtained from this labeling according to the induction hypothesis.  It also follows that each mutation sequence within $\mu_{\mathcal{C}}$ is well-defined.  Next, we show that $\mu_Q$ is a maximal green sequence for $Q$.  

Let $T_{Q}$ be the $3$-cycle belonging to $Q$ but not $\mathcal{C}$. It follows that $T_{Q} \in R_{1}$ of $Q$.  We begin by mutating in $L_{1_{Q}}$, where $L_{1_{Q}}, F_{1_{Q}}, S_{1_{Q}} \in T_{Q}$. The diagram below on the left depicts $\mu_{L_{1_{Q}}}(\hat{Q})$.

    \begin{minipage}{.24\textwidth}
      \begin{tikzpicture}[
> = stealth, 
            shorten < = 4pt, ]
            
\node[draw=none, fill=none] at (1.6,0.2) {$\mathcal{C}$};
\node[draw, shape=circle, fill=black, scale=0.5] at (2,-1) { };
\node[draw=none, fill=none] at (1.7,-1.2) {$S_{1_{Q}}$};
\node[draw, shape=circle, fill=black, scale=0.5] at (1.4,-1.7) { };
\node[draw=none, fill=none] at (1.1,-1.9) {$F_{1_{Q}}$};
\node[draw, shape=circle, fill=black, scale=0.5] at (2.6,-1.7) { };
\node[draw=none, fill=none] at (2.9, -1.5) {$L_{1_{Q}}$};
\node[draw=none, fill=none] at (7, 3) { };
\node[draw=none, fill=none] at (7, -3) { };
\node[draw=none, fill=none] at (2.2,0) { };

\draw [black] plot [smooth cycle] coordinates {(0,0) (1,1) (3,1) (3,0) (2,-1)};
\draw [<-] (1.4,-1.7) -- node { } (2.6,-1.7);
\draw [<-] (2.6,-1.7) -- node { } (2,-1);

\end{tikzpicture}
\end{minipage}%
                      \begin{minipage}{0.5\textwidth}
\centerline{$\xrightarrow{\makebox[2cm]{$\mu_{\mathcal{C}}$}}$}
\end{minipage}%
 \begin{minipage}{.1\textwidth}
 
 \begin{tikzpicture}[
> = stealth, 
            shorten < = 4pt, ]
            
\node[draw=none, fill=none] at (1.6,0.2) {$\mathcal{C}$};
\node[draw, shape=circle, fill=black, scale=0.5] at (2,-1) { };
\node[draw=none, fill=none] at (1.7,-1.2) {$S_{1_{Q}}$};
\node[draw, shape=circle, fill=black, scale=0.5] at (1.4,-1.7) { };
\node[draw=none, fill=none] at (1.1,-1.9) {$F_{1_{Q}}$};
\node[draw, shape=circle, fill=black, scale=0.5] at (2.6,-1.7) { };
\node[draw=none, fill=none] at (2.9, -1.5) {$L_{1_{Q}}$};
\node[draw=none, fill=none] at (7, 3) { };
\node[draw=none, fill=none] at (7, -3) { };

\draw [black] plot [smooth cycle] coordinates {(0,0) (1,1) (3,1) (3,0) (2,-1)};
\draw [<-] (1.4,-1.7) -- node { } (2.6,-1.7);
\draw [dashed,<-] (2.2,0) -- node { } (2.6,-1.7);

\end{tikzpicture}
\end{minipage}%

Observe that the quiver $\mu_{L_{1_Q}}(\hat{Q})$ satisfies the conditions of Lemma \ref{2.14}.  By induction $\mu_{\mathcal{C}}$ is a maximal green sequence for $\mathcal{C}$, hence part 1 of this lemma implies that all vertices of $\mathcal{C}$ are red in $\mu_{\mathcal{C}}\mu_{L_{1_Q}}(\hat{Q})$.  Moreover, by Lemma \ref{2.14} part 2, the resulting quiver $\mu_{\mathcal{C}}\mu_{L_{1_{1}}}(\hat{Q})$ is made up of $\mu_{\mathcal{C}}(\hat{\mathcal{C}})$ and the subquiver depicted below. 

\begin{center}
\begin{tikzpicture}[
            > = stealth, 
            shorten > = 1pt, 
            auto,
            node distance = 3cm, 
            semithick 
        ]

        \tikzstyle{every state}=[
            draw = black,
            thick,
            fill = white,
            minimum size = 1mm
        ]
     \node[draw=none, shape=circle, fill=black, scale=0.5] (1){};
     \node[draw=none, fill=none] at (0,0.5) {$F_{1_Q}$};
    \node[draw=none, shape=circle, fill=black, scale=0.5] (2) [below of =1]{};
     \node[draw=none, fill=none] at (0,-2.0) {$F'_{1_Q}$};
      \node[draw=none, shape=circle, fill=black, scale=0.5] (3)[right of = 1]{};
      \node[draw=none, fill=none] at (1.8,0.5) {$L_{1_Q}$};
     \node[draw=none, shape=circle, fill=black, scale=0.5] (4)[right of = 2]{};
     \node[draw=none, fill=none] at (1.8,-2.0) {$L'_{1_Q}$};
      
    \path[<-] (2) edge node { } (1);
        \path[->] (3) edge node { } (1);
            \path[<-] (2) edge node { } (1);
                \path[->] (1) edge node { } (4);
                    \path[->] (4) edge node { } (3);
  
  \end{tikzpicture}

\end{center}

In addition, we also know that $\mu_{\mathcal{C}}(\hat{\mathcal{C}})$ is connected to the quiver above by a single arrow starting at $L_{1_Q}$ and ending at some vertex of $\mathcal{C}_0$.  Therefore, we obtain the quiver $\mu_{\mathcal{C}}\mu_{L_{1_Q}}(\hat{Q})$ as shown above on the right.  Here, we omit the frozen vertices.  Also, to emphasize the connection between $\mu_{\mathcal{C}}(\hat{\mathcal{C}})$ and the above quiver we drew a dotted arrow between a vertex in $\mathcal{C}$ and $L_{1_{Q}}$.

In $\mu_{\mathcal{C}}\mu_{L_{1_{Q}}}(\hat{Q})$ the vertex $L_{1_Q}$ is red, because it was red prior to the mutation sequence $\mu_{\mathcal{C}}$, and all vertices of $\mathcal{C}$ are also red.  Therefore, the only green vertex in $\mu_{\mathcal{C}}\mu_{L_{1_{Q}}}(\hat{Q})$ is $F_{1_Q}$.  

Recall, that $\mu_{Q}$, the mutation sequence for $Q$ described in the theorem, can be written as 
$$\mu_{L_{1_{Q}}}\mu_{F_{1_{Q}}}\mu_{\mathcal{C}}\mu_{L_{1_{Q}}}.$$
Hence, the next step is to mutate in $F_{1_{Q}}$, which by the above is a green mutation.  The resulting quiver is shown below on the left. It is clear that since $F_{1_{Q}}$ has an arrow to each of $L_{1_{Q}}'$ and $F_{1_{Q}}'$ in $\mu_{\mathcal{C}}\mu_{L_{1_{Q}}}(\hat{Q})$, mutating in this vertex will eliminate the arrow $L_{1_{Q}}' \rightarrow L_{1_{Q}}$ and add an arrow $L_{1_{Q}} \rightarrow F_{1_{Q}}'$.  It follows that $L_{1_{Q}}$ is now green with an arrow $L_{1_Q} \rightarrow F_{1_{Q}}'$.   

The final step according to the procedure is $\mu_{L_{1_{Q}}}$, and the quiver $\mu_{L_{1_{Q}}}\mu_{F_{1_{Q}}}\mu_{\mathcal{C}}\mu_{L_{1_{Q}}}(\hat{Q})$ is shown below on the right. Since, there exists an arrow $L_{1_{Q}}\rightarrow F_{1_{Q}}'$, the arrow $F_{1_{Q}}' \rightarrow F_{1_{Q}}$ is cancelled. Therefore, $F_{1_{Q}}$ remains red.    

    \begin{minipage}{.24\textwidth}
      \begin{tikzpicture}[
> = stealth, 
            shorten < = 4pt, ]
            
\node[draw=none, fill=none] at (1.6,0.2) {$\mathcal{C}$};
\node[draw, shape=circle, fill=black, scale=0.5] at (2,-1) { };
\node[draw=none, fill=none] at (1.7,-1.2) {$S_{1_{Q}}$};
\node[draw, shape=circle, fill=black, scale=0.5] at (1.4,-1.7) { };
\node[draw=none, fill=none] at (1.1,-1.9) {$F_{1_{Q}}$};
\node[draw, shape=circle, fill=black, scale=0.5] at (2.6,-1.7) { };
\node[draw=none, fill=none] at (2.9, -1.5) {$L_{1_{Q}}$};
\node[draw=none, fill=none] at (7, 3) { };
\node[draw=none, fill=none] at (7, -3) { };

\draw [black] plot [smooth cycle] coordinates {(0,0) (1,1) (3,1) (3,0) (2,-1)};
\draw [<-] (2.6,-1.7) -- node { } (1.4,-1.7);
\draw [dashed,<-] (2.2,0) -- node { } (2.6,-1.7);

\end{tikzpicture}
\end{minipage}%
                      \begin{minipage}{0.5\textwidth}
\centerline{$\xrightarrow{\makebox[2cm]{$\mu_{L_{1_{Q}}}$}}$}
\end{minipage}%
 \begin{minipage}{.1\textwidth}
 
 \begin{tikzpicture}[
> = stealth, 
            shorten < = 4pt, ]
            
\node[draw=none, fill=none] at (1.6,0.2) {$\mathcal{C}$};
\node[draw, shape=circle, fill=black, scale=0.5] at (2,-1) { };
\node[draw=none, fill=none] at (1.7,-1.2) {$S_{1_{Q}}$};
\node[draw, shape=circle, fill=black, scale=0.5] at (1.4,-1.7) { };
\node[draw=none, fill=none] at (1.1,-1.9) {$F_{1_{Q}}$};
\node[draw, shape=circle, fill=black, scale=0.5] at (2.6,-1.7) { };
\node[draw=none, fill=none] at (2.9, -1.5) {$L_{1_{Q}}$};
\node[draw=none, fill=none] at (7, 3) { };
\node[draw=none, fill=none] at (7, -3) { };
\node[draw=none, fill=none] at (2.2,0) { };

\draw [black] plot [smooth cycle] coordinates {(0,0) (1,1) (3,1) (3,0) (2,-1)};
\draw [<-] (1.4,-1.7) -- node { } (2.6,-1.7);
\draw [dashed,->] (1.4,-1.7) to [out=120,in=180] (2.2,0);
\draw [dashed,<-] (2.6,-1.7)  -- node { } (2.2,0);

\end{tikzpicture}
\end{minipage}%

Observe that the last two mutations in $F_{1_{Q}}$ and $L_{1_{Q}}$ did not change the color of any vertices in $\mathcal{C}$ back to green.  Since prior to these mutations all arrows went from $L_{1_{Q}}$ to vertices of $\mathcal{C}$, hence no new arrows could be added between vertices of $\mathcal{C}$ and the frozen vertices $V'$ such that $V\not\in\mathcal{C}_0$. It is also clear that any connection to $\mathcal{C}$ has no impact on the color of $F_{1_{Q}}$ and $L_{1_{Q}}$ after the mutations $\mu_{\mathcal{C}}\mu_{L_{1_Q}}$. It follows that every vertex in $Q$ is red after performing $\mu_Q$, hence $\mu_Q (\hat{Q})\cong \check{Q}$.  Moreover, we obtained the quiver with all red vertices via the procedure outlined in the theorem. By induction, the procedure holds for an arbitrary number of $3$-cycles.  

As a final note, observe that in $\mu_Q$ the mutations in various cycles within the same region commute.  By induction the mutation sequences $\underline{\mu_i}$ and $\mu_{L_i}$ are well-defined for all $i = 2, \dots, m'$.   In the beginning we saw that $\mu_{L_1}$ is also well-defined.  Finally, every 3-cycle in the outermost region of $Q$ will be in an analogous configuration relative to the rest of the quiver as $T_{1_Q}$ was relative to $\mathcal{C}$.  Hence, we can see that $\mu_{L_{1_j}}\mu_{F_{1_j}}$ and $\mu_{L_{1_k}}\mu_{F_{1_k}}$ will commute for every $j$ and $k$.  This completes the proof of the theorem.  
\end{proof}

\begin{remark} \label{5.20}
The procedure in Theorem \ref{5.9} produces an MGS in $n+t$ mutations, where $n =  | Q_{0} |$ and $t$ is the number of $3$-cycles in $Q$. Every vertex must trivially be mutated at least once in any procedure that produces an MGS. Moreover, the procedure mutates each $L_{i_{j}}$ twice. There is one $L_{i_{j}}$ per $3$-cycle. It follows that the number of mutations is $n+t$. 
\end{remark}

Consider the following example that illustrates the procedure in Theorem \ref{5.9}.  

\begin{ex}
Suppose we have the configuration of interior triangles shown below on the left.  Let us find a maximal green sequence for this quiver as outlined in Theorem \ref{5.9}.  We begin by labeling the regions and the corresponding vertices.  \\

\begin{tikzpicture}[
> = stealth, 
shorten > = 1pt, 
auto,
node distance = 1.35cm, 
semithick 
]

\tikzstyle{every state}=[
draw = black,
thick,
fill = white,
minimum size = 3mm
]

\node[draw=none, fill=none] (1) {$1$};

\node[draw=none, fill=none] (2) [right of=1] {$2$};

\node[draw=none, fill=none] (3) [above of=2] {$3$};

\node[draw=none, fill=none] (4) [right of=3] {$4$};

\node[draw=none, fill=none] (5) [above of=4] {$5$};

\node[draw=none, fill=none] (6) [right of=2] {$6$};

\node[draw=none, fill=none] (7) [right of=6] {$7$};

\node[draw=none, fill=none] (8) [right of=5] {$8$};

\node[draw=none, fill=none] (9) [above of=8] {$9$};

\node[draw=none, fill=none] (10) [right of=8] {$10$};

\node[draw=none, fill=none] (11) [above of=10] {$11$};

\node[draw=none, fill=none] (12) [right of=10] {$12$};

\node[draw=none, fill=none] (13) [above of=12] {$13$};

\path[->] (1) edge node { } (2);

\path[->] (2) edge node { } (3);

\path[->] (3) edge node { } (1);

\path[->] (3) edge node { } (4);

\path[->] (4) edge node { } (5);

\path[->] (5) edge node { } (3);

\path[->] (4) edge node { } (6);

\path[->] (6) edge node { } (7);

\path[->] (7) edge node { } (4);

\path[->] (5) edge node { } (8);

\path[->] (8) edge node { } (9);

\path[->] (9) edge node { } (5);

\path[->] (8) edge node { } (10);

\path[->] (10) edge node { } (11);

\path[->] (11) edge node { } (8);

\path[->] (13) edge node { } (10);

\path[->] (10) edge node { } (12);

\path[->] (12) edge node { } (13);

\end{tikzpicture} \begin{tikzpicture}[
> = stealth, 
shorten > = 1pt, 
auto,
node distance = 1.35cm, 
semithick 
]

\tikzstyle{every state}=[
draw = black,
thick,
fill = white,
minimum size = 3mm
]

\node[draw=none, fill=none] (1) {$1$};

\node[draw=none, fill=none] (2) [right of=1] {$2$};

\node[draw=none, fill=none] (3) [above of=2] {$3$};

\node[draw=none, fill=none] (4) [right of=3] {$F_{4_1}$};

\node[draw=none, fill=none] (5) [above of=4] {$L_{4_1}$};

\node[draw=none, fill=none] (6) [right of=2] {$F_{1_2}$};

\node[draw=none, fill=none] (7) [right of=6] {$L_{1_2}$};

\node[draw=none, fill=none] (8) [right of=5] {$F_{3_1}$};

\node[draw=none, fill=none] (9) [above of=8] {$L_{3_1}$};

\node[draw=none, fill=none] (10) [right of=8] {$F_{2_1}$};

\node[draw=none, fill=none] (11) [above of=10] {$L_{2_1}$};

\node[draw=none, fill=none] (12) [right of=10] {$F_{1_1}$};

\node[draw=none, fill=none] (13) [above of=12] {$L_{1_1}$};

\path[->] (1) edge node {\,\,\,$R_m$} (2);

\path[->] (2) edge node { } (3);

\path[->] (3) edge node { } (1);

\path[->] (3) edge node {\,\,\,\;\;$R_4$ } (4);

\path[->] (4) edge node { } (5);

\path[->] (5) edge node { } (3);

\path[->] (4) edge node { } (6);

\path[->] (6) edge node {$R_1$\,\,} (7);

\path[->] (7) edge node { } (4);

\path[->] (5) edge node { $\;\;R_3$} (8);

\path[->] (8) edge node { } (9);

\path[->] (9) edge node { } (5);

\path[->] (8) edge node { \,\,\,$R_2$} (10);

\path[->] (10) edge node { } (11);

\path[->] (11) edge node { } (8);

\path[->] (13) edge node { } (10);

\path[->] (10) edge node {\,\,$R_1$} (12);

\path[->] (12) edge node { } (13);

\end{tikzpicture}\\

First, choose the innermost region to be the triangle 123. We proceed to label the leaders and followers in each region accordingly, see the quiver on the right.  According to the procedure, we begin by mutating at the leaders of each consecutive region, yielding the green sequence $\mu_5 \mu_9 \mu_{11} \mu_7 \mu_{13}$.   Now we turn all of the vertices in the innermost region red by following a single triangle procedure as in Example \ref{5.1}.  Hence, we obtain four additional mutations $\mu_3 \mu_2 \mu_1 \mu_3$. Moving outward from $R_4$ to $R_1$, we turn the vertices of each triangle red in two steps, first by mutating at the followers and then by mutating at the leaders a second time. This gives us the MGS 
$$\mu_7 \mu_6 \mu_{13} \mu_{12} \mu_{11} \mu_{10} \mu_9 \mu_8 \mu_5 \mu_4 \;\;\; \mu_3 \mu_2 \mu_1 \mu_3 \;\;\; \mu_5 \mu_9 \mu_{11} \mu_7 \mu_{13}.$$
Notice, that this MGS has length $19=13+6=n+t$, as desired.
\end{ex}

\section{General Case}

In this section we describe a procedure that yields a maximal green sequence for an arbitrary type $\mathbb{A}$ quiver.  

\begin{defn}
 A \textit{non-isolating vertex} corresponding to a 3-cycle configuration $\mathcal{C}$ is a vertex $i \not \in \mathcal{C}_0$  such that there exists an arrow from $i$ to a vertex in $\mathcal{C}$.   An \textit{isolating vertex} corresponding to $\mathcal{C}$ is a vertex $j\not \in \mathcal{C}_0$ such that there exists an arrow going to $j$ from a vertex in $\mathcal{C}$. 
 \end{defn}
 
\begin{defn}
 A 3-cycle configuration $\mathcal{C}$ is called \textit{isolated} if each vertex $j\not \in \mathcal{C}_0$ connected to $\mathcal{C}$ by an arrow is an isolating vertex. 
 \end{defn}
 
\begin{defn} Two triangle configurations are \emph{connected by a fan} if there exists a fan containing both an isolating and a non-isolating vertex corresponding to these configurations, as shown below.  Call such a fan a \textit{connecting fan}.
\end{defn}

\begin{center} 
\begin{tikzpicture}[scale=0.5]
[ 
> = stealth, 
            shorten > = 1pt, 
            auto,
            node distance = 2.5cm, 
            semithick 
        ]

        \tikzstyle{every state}=[
            draw = black,
            thick,
            fill = white,
            minimum size = 6mm
        ]
				
				\node[draw=none, fill=none] (1) {$\bullet$};
				\node[draw=none, fill=none] (2) [above right of =1] {$\bullet$};
			   \node[draw=none, fill=none] (3) [right =1] {$\bullet$};
				\node[draw=none, fill=none] (4) [right =3, xshift=-.8cm] {$\bullet$};
        \node[draw=none, fill=none] (5) [right =4, xshift=-.60cm] {$\bullet$};
        \node[draw=none, fill=none] (6) [right =5, xshift=-.4cm] {$\cdots$};
        \node[draw=none, fill=none] (7) [right =6, xshift=-.2cm] {$\bullet$};
        \node[draw=none, fill=none] (8) [right =7] {$\bullet$};
				\node[draw=none, fill=none] (9) [above right of =8] {$\bullet$};
				\node[draw=none, fill=none] (10) [right =8, xshift=.2cm] {$\bullet$};
				
				
				\node at (8.5, -1) (13) {$\underbrace{\hspace{4cm}}_{\text{\small{connecting fan}}}$};
				\node[] at (4.8,3) (11) {\small{isolating}};
				\node[] at (12,3) (12) {\small{non-isolating}};
			                                
				\path[->] (1) edge node { } (2);	
				\path[->] (2) edge node { } (3);														
				\path[->] (3) edge node { } (1);														
				\path[->] (3) edge node { } (4);															
				\path[->] (4) edge node { } (5);															
				\path[->] (5) edge node { } (6);
				\path[->] (6) edge node { } (7);
				\path[->] (7) edge node { } (8);
			  \path[->] (8) edge node { } (9);
				\path[->] (9) edge node { } (10);
				\path[->] (10) edge node { } (8);

\path[->] (11) edge node { } (4);	
\path[->] (12) edge node { } (7);

\end{tikzpicture}
\end{center}

Similarly, two triangle configurations are \emph{connected by an arrow} if there exists an arrow from one triangle configuration to the other.  Call such an arrow a \emph{connecting arrow}. 

\begin{center} 
\begin{tikzpicture}[scale=0.9]
[ 
> = stealth, 
            shorten > = 1pt, 
            auto,
            node distance = 2.5cm, 
            semithick 
        ]

        \tikzstyle{every state}=[
            draw = black,
            thick,
            fill = white,
            minimum size = 6mm
        ]
				
				\node[draw=none, fill=none] (1) {$\bullet$};
				\node[draw=none, fill=none] (2) [above right of =1] {$\bullet$};
			   \node[draw=none, fill=none] (3) [right =1, xshift=-.14cm] {$\bullet$};
				\node[draw=none, fill=none] (4) [right =3, xshift=-.99cm] {$\bullet$};
        \node[draw=none, fill=none] (5) [above right of =4] {$\bullet$};
				\node[draw=none, fill=none] (6) [right of =4] {$\bullet$};
			            \node at (1.85, -.5) (13) {$\underbrace{\hspace{1.3cm}}_{\text{\small{connecting arrow}}}$};

				\path[->] (1) edge node { } (2);	
				\path[->] (2) edge node { } (3);														
				\path[->] (3) edge node { } (1);														
				\path[->] (3) edge node { } (4);															
				\path[->] (4) edge node { } (5);															
				\path[->] (5) edge node { } (6);
				\path[->] (6) edge node { } (4);

\end{tikzpicture}
\end{center}

\begin{lem} \label{6.1}
Let $Q$ be a quiver containing a 3-cycle configuration $\mathcal{C}$ such that there is no connecting fan or a connecting arrow attached to any non-isolating vertices corresponding to $\mathcal{C}$. Then $\mathcal{C}$ can be isolated via green mutations at sources.
\end{lem}

\begin{proof} 
We must only consider the non-isolating vertices corresponding to $\mathcal{C}$.  Suppose a non-isolating vertex $i$ that has an arrow into $\mathcal{C}$ is a source. Mutating at this vertex will be a green source mutation. If $i$ is not a source, then by assumption $i$ cannot belong to some other 3-cycle configuration.  Hence, there is an arrow $j \rightarrow i$, where $j$ belongs to some zigzag-fan configuration. In that case, if $j$ is a source mutating at $j$ will turn $i$ into a source. Hence, $i$ can also be mutated afterwards.  

Finally, if $j$ is not a source then $i$ is part of a larger fan $F$.  Moreover, the other end of the fan $F$ will be a source, since it cannot be a connecting fan by assumption. Thus, a fan procedure can be applied to $F$, which dictates to mutate at every vertex in $F$ once, and every mutation is a source mutation.  The last mutation in this sequence will be at $i$, making it an isolating vertex.  In this way all non-isolating vertices can been mutated, making $\mathcal{C}$ isolated. 
\end{proof}

Next we describe a way to obtain a maximal green sequence for an arbitrary quiver mutation equivalent to type $\mathbb{A}$.  The procedure relies on isolating each 3-cycle configuration via the construction in Lemma \ref{6.1}, and then applying the 3-cycle procedure.  Recall that if $F$ is a fan shown below.  
$$\xymatrix{i_1 \ar[r] & i_2 \ar[r] & \cdots \ar[r] & i_{r-1} \ar[r] & i_r}$$
Then by Proposition \ref{4.1}, the fan procedure $\mu_F = \mu_{i_1}\mu_{i_2} \cdots \mu_{i_r}$ is a maximal green sequence for this quiver.  

Let $\mathcal{C}$ be a 3-cycle configuration in $Q$, by $\text{Ni}(\mathcal{C})$ we denote the set of all non-isolating vertices corresponding to $\mathcal{C}$.  

\begin{thm}\label{6.4}
Consider an arbitrary quiver $Q$ arising as a triangulation of a polygon. The following procedure yields a maximal green sequence for $Q$.
\vspace{11pt}
\begin{enumerate}[1.]
\item  Let $\mathcal{C}_0$ be a 3-cycle configuration such that no $v_j \in \textup{Ni}(\mathcal{C}_0)$ is also a part of a connecting fan or a connecting arrow.  For every such $v_j$ there exists a fan $F_j$ in $Q$ containing $v_j$.  Let $\mu_{F_j}$ denote the fan procedure on $F_j$.  Perform the following mutation sequence $\mu_{\mathcal{C}_0}\mu_{F^0}$ on $\hat{Q}$, where

\vspace{11pt}
$\hspace{3cm}\mu_{F^0}: \hspace{.5cm}\textup{the composition of all}\; \mu_{F_j} \;\textup{for every } v_j \in \textup{Ni}(\mathcal{C}_0).$

$\hspace{3cm}\mu_{\mathcal{C}_0}\,: \hspace{.52cm} \textup{the 3-cycle procedure for}\; \mathcal{C}_0.$

\vspace{11pt}
\noindent Let $Q^1$ be the full subquiver obtained from $Q$ by deleting the mutated vertices. 

$$Q^1 = (Q\setminus\mathcal{C}_0) \setminus \bigcup_{v_j\in \textup{Ni}(\mathcal{C}_0)} F_j$$

\item Perform step 1 for the quiver $Q^1$, to obtain a new quiver $Q^2$.  Repeat inductively to generate $Q^i$, where 

$$Q^i = (Q^{i-1}\setminus\mathcal{C}_{i-1}) \setminus \bigcup_{v_j\in \textup{Ni}(\mathcal{C}_{i-1})} F_j$$

and the corresponding mutation sequences $\mu_{\mathcal{C}_i}\mu_{F^{i}}$, where  

\vspace{11pt}
$\hspace{3cm}\mu_{F^i}: \hspace{.5cm}\textup{the composition of all}\; \mu_{F_j} \;\textup{for every } v_j \in \textup{Ni}(\mathcal{C}_i).$

$\hspace{3cm}\mu_{\mathcal{C}_i}\,: \hspace{.52cm} \textup{the 3-cycle procedure for}\; \mathcal{C}_i.$

\vspace{11pt}
\noindent In this way obtain a quiver $Q^m$ that does not contain any 3-cycles.  Hence, $Q^m$ is a possibly disjoint union of acyclic quivers.  Perform the zigzag-fan procedure $\mu_m$ on $Q^m$.  
\end{enumerate}

\vspace{11pt}
\noindent The resulting maximal green sequence is 
$$\mu_m\;\;\mu_{\mathcal{C}_{m-1}}\mu_{F^{m-1}}\;\;\cdots\;\;\mu_{\mathcal{C}_1}\mu_{F^1}\;\;\mu_{\mathcal{C}_0}\mu_{F^0}.$$

\end{thm}

\begin{proof} 
If $Q$ does not contain any 3-cycle configuration then $Q = Q^m$, and an MGS for $Q$ is given by a zigzag-fan procedure.  Therefore, the theorem above holds.

Given a quiver $Q$ that contains 3-cycle configurations, there must be at least one 3-cycle configuration that does not have a non-isolating vertex that is also a part of a connecting fan or a connecting arrow.  This follows from the fact that given a chain of $s$ 3-cycle configurations connected by fans or arrows, there can be no more than $s-1$ connecting fans or arrows. Thus, there cannot be a connecting fan or a connecting arrow starting at each of the $s$ 3-cycle configurations.  This shows that $\mathcal{C}_0$ exists.  

By Lemma \ref{6.1} the sequence of mutations $\mu_{F^0}$ isolates $\mathcal{C}_0$.  Then, by Lemma \ref{2.14} and Remark \ref{2.15}, we can perform the 3-cycle procedure on $\mathcal{C}_0$ making all vertices in this 3-cycle red, while preserving the color of vertices in each $F_j$.  Moreover, by part 2 of the same lemma we see that the mutation sequence $\mu_{\mathcal{C}_0}\mu_{F^0}$ does not affect the vertices outside of $\mathcal{C}_0$.  Hence,  in $\mu_{\mathcal{C}_0}\mu_{F^0}(\hat{Q})$ if a red vertex $x$ is connected to a green vertex $y$ , then there exists an arrow $y\rightarrow x$.  Since, we will not mutate at these red vertices a source mutation in $Q^1$ corresponds to a source mutation in $\mu_{\mathcal{C}_0}\mu_{F^0}(\hat{Q})$.  This together with Lemma \ref{2.14} implies that the subsequent mutations $\mu_{\mathcal{C}_1}\mu_{F^1}\mu_{\mathcal{C}_0}\mu_{F^0}(\hat{Q})$ will not turn any of the red vertices in $\mu_{\mathcal{C}_0}\mu_{F^0}(\hat{Q})$ back to green.  

Applying the same reasoning, we see that with every additional mutation step $\mu_{\mathcal{C}_i}\mu_{F^i}$ the red vertices of $\mu_{\mathcal{C}_{i-1}}\mu_{F^{i-1}}\cdots \mu_{\mathcal{C}_0}\mu_{F^0}(\hat{Q})$ remain red, and all vertices in $\mathcal{C}_i$ and the corresponding fans also become red.  

Therefore, it remains to show that the procedure will terminate yielding an acyclic quiver $Q^m$.  But this follows because $Q$ is a finite quiver, and at every step we remove one configuration of 3-cycles together with the corresponding fans.  This shows that the procedure yields a quiver with all red vertices.
\end{proof}

\begin{remark}\label{6.20}
 The procedure outlined above constructs a maximal green sequence for $Q$ of length $n+t$, where $n$ is the number of vertices in $Q$ and $t$ is the number of 3-cycles in $Q$.  It dictates to mutate $n'+t'$ times for every 3-cycle configuration $\mathcal{C}$, where $n'$ is the number of vertices in $\mathcal{C}$ and $t'$ is the number of 3-cycles in $\mathcal{C}$.  See Remark \ref{5.20}.  On the other hand it requires to mutate at every vertex outside of a 3-cycle configuration exactly one, hence obtaining the desired $n+t$ steps.  
\end{remark}

\begin{ex}

We will demonstrate the procedure outlined in Theorem \ref{6.4} for the following quiver $Q.$
\begin{center} 
\begin{tikzpicture}[scale=0.5]
[ 
> = stealth, 
            shorten > = 1pt, 
            auto,
            node distance = 4.5cm, 
            semithick 
        ]

        \tikzstyle{every state}=[
            draw = black,
            thick,
            fill = white,
            minimum size = 6mm
        ]
				
				\node[draw=none, fill=none] (1) {$4$};
				\node[draw=none, fill=none] (2) [above right of =1] {$5$};
			   \node[draw=none, fill=none] (3) [right =1] {$6$};
				\node[draw=none, fill=none] (4) [right =3, xshift=-.8cm] {$9$};
        \node[draw=none, fill=none] (5) [right =4, xshift=-.60cm] {$10$};
        \node[draw=none, fill=none] (6) [right =5, xshift=-.4cm] {$11$};
        \node[draw=none, fill=none] (7) [right =6, xshift=-.2cm] {$12$};
        \node[draw=none, fill=none] (8) [right =7] {$13$};
				\node[draw=none, fill=none] (9) [above right of =8] {$14$};
				\node[draw=none, fill=none] (10) [right =8, xshift=.2cm] {$15$};
				\node[draw=none, fill=none] (11) [left of =1] {$8$};
				\node[draw=none, fill=none] (12) [left of =11] {$7$};
				\node[draw=none, fill=none] (13) [left of =2] {$3$};
				\node[draw=none, fill=none] (14) [left of =13] {$2$};
				\node[draw=none, fill=none] (15) [left of =14] {$1$};
				\node[draw=none, fill=none] (16) [left of =9] {$22$};
				\node[draw=none, fill=none] (17) [left of =16] {$23$};
				\node[draw=none, fill=none] (18) [above right of =10] {$16$};
				\node[draw=none, fill=none] (19) [right of =10] {$17$};
				\node[draw=none, fill=none] (20) [above right of =18] {$24$};
				\node[draw=none, fill=none] (21) [right of =18] {$25$};
				\node[draw=none, fill=none] (22) [below of =19] {$18$};
				\node[draw=none, fill=none] (23) [ right of =22] {$19$};
				\node[draw=none, fill=none] (24) [left of =22] {$20$};
				\node[draw=none, fill=none] (25) [left of =24] {$21$};
				
				\path[->] (1) edge node { } (2);	
				\path[->] (2) edge node { } (3);	
				\path[->] (3) edge node { } (1);	
				\path[->] (13) edge node { } (2);	
				\path[->] (14) edge node { } (13);	
				\path[->] (15) edge node { } (14);	
				\path[->] (11) edge node { } (1);	
				\path[->] (11) edge node { } (12);	
				\path[->] (3) edge node { } (4);	
				\path[->] (4) edge node { } (5);	
				\path[->] (5) edge node { } (6);	
				\path[->] (6) edge node { } (7);	
				\path[->] (7) edge node { } (8);	
				\path[->] (8) edge node { } (9);	
				\path[->] (9) edge node { } (10);	
				\path[->] (10) edge node { } (8);	
				\path[->] (9) edge node { } (16);	
				\path[->] (16) edge node { } (17);	
				\path[->] (10) edge node { } (18);	
				\path[->] (18) edge node { } (19);	
				\path[->] (19) edge node { } (10);	
				\path[->] (18) edge node { } (21);	
				\path[->] (21) edge node { } (20);	
				\path[->] (20) edge node { } (18);	
				\path[->] (19) edge node { } (22);	
				\path[->] (22) edge node { } (23);	
				\path[->] (23) edge node { } (19);	
				\path[->] (25) edge node { } (24);	
				\path[->] (22) edge node { } (24);

\end{tikzpicture}
\end{center}

Below is a maximal green sequence for $Q$ obtained by applying Theorem \ref{6.4}.   Note that its length is $30,$ which is equal to $n+t,$ as desired. 

 $$\mu_{23} \mu_{22} \mu_{20} \mu_{21} \mu_{7}\;\;\mu_{19} \mu_{18} \mu_{24} \mu_{25} \mu_{17} \mu_{16} \mu_{15} \mu_{14} \mu_{13} \mu_{15} \mu_{17} \mu_{24} \mu_{19} \;\; \mu_{12} \mu_{11} \mu_{10} \mu_{9}\;\;\mu_4 \mu_6 \mu_5 \mu_4 \;\; \mu_3 \mu_2 \mu_1 \mu_8.$$ 

This sequence is constructed in the following way.  For $Q$, we only have one choice of $\mathcal{C}_0,$ the 3-cycle configuration composed of the vertices 4, 5, and 6. We see that $3$ and $8$ are non-isolating vertices corresponding to $\mathcal{C}_0.$ Vertex $8$ is a source, so we mutate at $8$ immediately. However, $3$ is part of a fan, so we perform the fan procedure, starting at 1, then mutating at 2, and finishing at $3.$  In this way, the 3-cycle $\mathcal{C}_0$ becomes isolated, and so we perform the 3-cycle procedure. The full resulting mutation sequence is $\mu_4 \mu_6 \mu_5 \mu_4 \mu_3 \mu_2 \mu_1 \mu_8.$ 

Next, we consider $Q^1$ consisting of vertices $Q_0\setminus \{1,2,3,4,5,6,8\}$.  There is only one 3-cycle configuration in $Q^1$, composed of the vertices $13,$ $14,$ $15,$ $16,$ $17,$ $18,$ $19,$ $24,$ and $25.$ This is $\mathcal{C}_1$.  Its only non-isolating vertex is $12,$ which is the endpoint of a fan starting at 9.  Notice that this fan is not a connecting fan, since $\mathcal{C}_0$ is not present in $Q^1.$ We first perform the fan procedure on this fan, concluding with the mutation at $12.$ We then perform the 3-cycle procedure on $\mathcal{C}_1.$ The full resulting mutation sequence will be $\mu_{19} \mu_{18} \mu_{24} \mu_{25} \mu_{17} \mu_{16} \mu_{15} \mu_{14} \mu_{13} \mu_{15} \mu_{17} \mu_{24} \mu_{19} \mu_{12} \mu_{11} \mu_{10} \mu_{9}$. 

The quiver $Q^2$ will be composed of the vertices $7,$ $20,$ $21,$ $22,$ and $23.$ Note that it is acyclic, and so $Q^2=Q^m.$ Performing the zigzag-fan procedure on $Q^m$ results in the mutation sequence $\mu_{23} \mu_{22} \mu_{20} \mu_{21} \mu_{7}$.
\end{ex}

\section{Minimality}

As mentioned earlier the minimal length of an MGS for any acyclic quiver equals the number of vertices in the quiver.  However, this is not the case for quivers containing oriented cycles.  

In this section, we use a geometric approach to find the length of the shortest MGS's for any type $\mathbb{A}$ quivers.  

\begin{lem} \label{03}
Suppose $T$ is a triangulation of a polygon $S$.  Consider any closed polygon $\mathfrak{p}$ composed of arcs in $T$ and inscribed in $S$.  Let $\mu_I$ be a maximal green sequence for $Q_T$.  Then there exists an arc $i\in \partial \mathfrak{p}$ that appears at least twice in $\mu_I$.  
\end{lem}

\begin{proof} 

Consider some polygon $\mathfrak{p}$ and its boundary $\partial \mathfrak{p}$ in the triangulation $T$.  Observe that each arc in $\mathfrak{p}$ is necessarily a part of some interior triangle.  To illustrate the proof consider the diagram below.  We know that no matter how the other arcs in $T\setminus \partial \mathfrak{p}$ are arranged and flipped, in order to obtain an MGS we will eventually have to flip an arc in $\partial \mathfrak{p}$. Without loss of generality, suppose $i$ is the first arc in $\partial \mathfrak{p}$ that will be flipped.  

\begin{center}
 \includegraphics[scale=0.4]{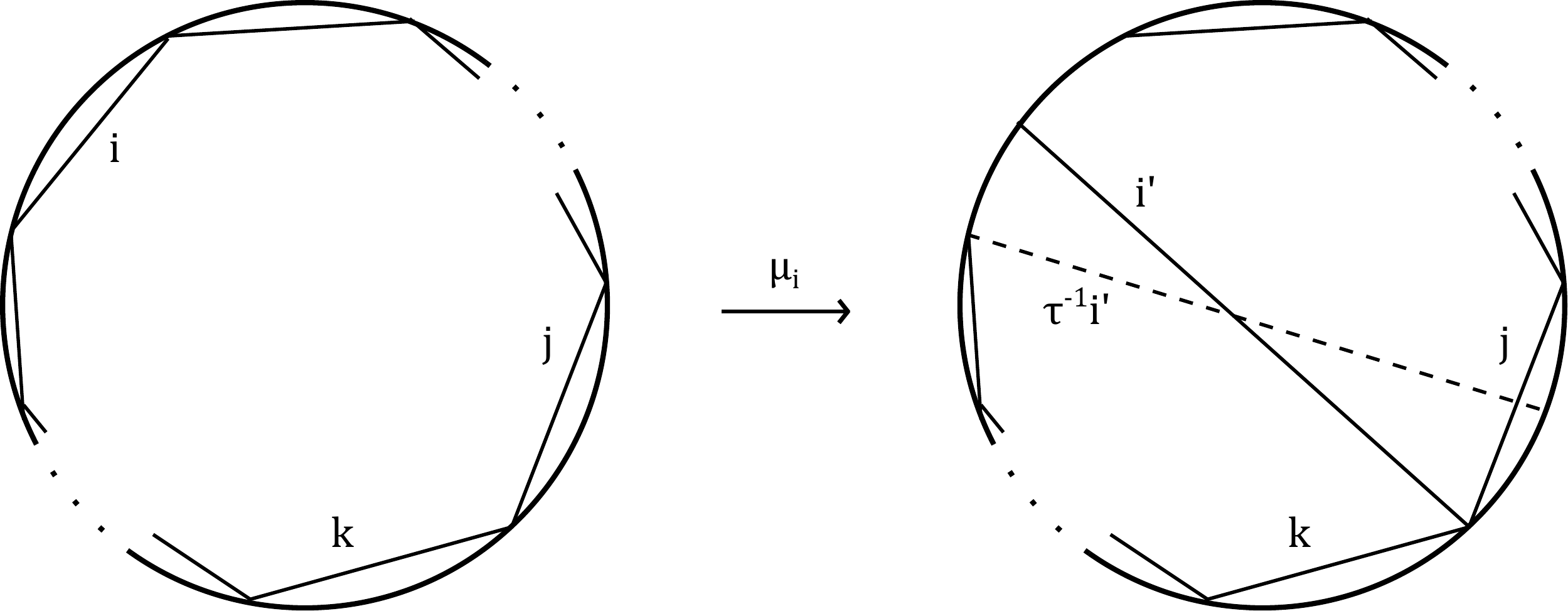}
\end{center}

Notice that if we flip $i$, the resulting arc $i'$ must have one endpoint in a vertex of $\mathfrak{p}$ where two other arcs in $\partial \mathfrak{p}$, say $j$ and $k$, meet.  Let $j\in \partial \mathfrak{p}$ be the arc lying counterclockwise from $i'$ with respect to the endpoint that $i',k$ and $j$ share.  Observe that if $i'$ is in its final position, meaning it does not have to be flipped again, then $i' \in \tau T$, by Proposition \ref{3.3}.  Since $\tau^{-1}\tau$ is the identity on the set of arcs in $S$, it means that $\tau^{-1}i' \in T$.  However, $\tau^{-1} i'$ intersects $j$ as illustrated above.  Recall that $i$ was the first arc in $\partial \mathfrak{p}$ to be flipped, hence $j\in T$.  This shows that $T$ contains $\tau^{-1}i'$ and $j$, which cannot happen since $T$ is a maximal collection of non-intersecting arcs.  Therefore, the arc $i'$ would have to flipped again.  This completes the proof of the lemma.

\end{proof}

\begin{thm} \label{04}
Suppose $T$ is a triangulation of a polygon.  Let $n$ be the number of arcs in $T$ and $t$ denote the number of distinct interior triangles in $T$.  Then any maximal green sequence for $Q_T$ is of length at least $n+t$.  
\end{thm}

\begin{proof}
Observe that if $T$ does not contain any interior triangles then the statement holds.  Otherwise, consider a closed polygon $\mathfrak{p}_1$ formed by all of the interior arcs in a triangulation $T$ such that $\mathfrak{p}_1$ is maximal.  This means that there is no other closed polygon $\mathfrak{q}$ composed of interior arcs such that $\mathfrak{p}_1$ is contained inside $\mathfrak{q}$.  In particular arcs in $\mathfrak{p}$ correspond to a triangle configuration in $T$.  By Lemma \ref{03} there exists an arc in $\partial \mathfrak{p}_1$ that has to be flipped at least twice. Call this arc $\gamma_1$.

Next, we consider the polygon(s) formed by eliminating the triangle $\Delta$ contained inside $\mathfrak{p}_1$ such that $\gamma_1 \in \Delta$.  Observe that this triangle $\Delta$ exists and is unique.  By the phrase eliminating this triangle we understand removing  
$$\Delta \cap \partial\mathfrak{p}_1$$
from $\mathfrak{p}_1$, which necessarily contains $\gamma_1$, and possibly another arc of $\Delta$.  In this way, if we remove two arcs of $\Delta$ then we obtain a single polygon $\mathfrak{p}_2$, and if we remove one arc of $\Delta$ then we obtain two polygons $\mathfrak{p}_2^1$ and $\mathfrak{p}_2^2$ connected at a single vertex. 

Either way, by Lemma \ref{03} it must be true that $\mathfrak{p}_2$ or each of $\mathfrak{p}_2^1$ and $\mathfrak{p}_2^2$ must (each) contain an arc that is flipped twice in a particular MGS. We also know that since we eliminated $\gamma_1$ this cannot be the same arc. 

Next, we eliminate the arcs that must be flipped twice in the second set of polygons $\mathfrak{p}_2$ or $\mathfrak{p}_2^1$ and $\mathfrak{p}_2^2$ in the same manner as above.  Hence, we form a third set of polygons. We can continue this process until we reach the last set of polygons, which will only consist of single triangles. Again by Lemma \ref{03} we know that since we have eliminated all arcs that must be flipped twice outside of these triangles, there must be an arc in each of these remaining triangles that must also be flipped twice. 

In this way we remove a single triangle each time we perform this process on any given polygon.  Since we started with $\mathfrak{p}_1$ being maximal we should have one arc that must be flipped twice for every triangle in $\mathfrak{p}_1$.  We can repeat the same procedure for every such maximal polygon in $T$. This implies that any MGS for $Q_T$ must be of length at least $n+t$.  
\end{proof}

\begin{cor}
Suppose $T$ is a triangulation of a polygon.  Let $n$ be the number of arcs in $T$ and $t$ denote the number of distinct interior triangles in $T$.  Then any maximal green sequence of minimal length for $Q_T$ is of length exactly $n+t$.  
\end{cor}

\begin{proof}
Theorem \ref{5.9} describes a procedure of constructing a maximal green sequence of length $n+t$, see Remark \ref{6.20}.  Applying Theorem \ref{04} yields the desired result. 
\end{proof} 

Observe that the length of a shortest MGS is independent of the way the 3-cycles are arranged in relation to each other.  It is interesting to note that this is not the case for the maximal length MGS's.  Consider the following example.  

\begin{ex}
It is known that the maximal length MGS for a quiver of type $\mathbb{A}$ without any 3-cycles is equal to $\frac{n(n+1)}{2}$, where $n$ is the number of nodes in the quiver, see \cite{BDP}. This does not hold for quivers composed of 3-cycles, however it does give us an upper bound.  The maximal length of an  MGS depends on the orientation of arrows between the shared vertices. Consider the quiver 

\begin{center} \begin{tikzpicture}[ 
> = stealth, 
            shorten > = 1pt, 
            auto,
            node distance = 1.5cm, 
            semithick 
        ]

        \tikzstyle{every state}=[
            draw = black,
            thick,
            fill = white,
            minimum size = 3mm
        ]
				\node[draw=none, fill=none] (1) {$1$};
        \node[draw=none, fill=none] (2) [above right of=1] {$2$};
        \node[draw=none, fill=none] (3) [right of=1] {$3$};
        \node[draw=none, fill=none] (4) [above right of=3] {$4$};
        \node[draw=none, fill=none] (5) [right of=3] {$5$};
        \node[draw=none, fill=none] (6) [above right of=5] {$6$};
        \node[draw=none, fill=none] (7) [right of=5] {$7$};
        \node[draw=none, fill=none] (8) [above right of=7] {$8$};
        \node[draw=none, fill=none] (9) [right of=7] {$9$};

        \path[->] (1) edge node { } (2);
				\path[->] (2) edge node { } (3);
				\path[->] (3) edge node { } (1);
				\path[->] (3) edge node { } (4);
				\path[->] (4) edge node { } (5);
				\path[->] (5) [ultra thick] edge node { } (3);
				\path[->] (5) edge node { } (6);
				\path[->] (6) edge node { } (7);
				\path[->] (7) [ultra thick] edge node { } (5);
				\path[->] (7) edge node { } (8);
				\path[->] (8) edge node { } (9);
				\path[->] (9) edge node { } (7);

\end{tikzpicture}
\end{center}

\noindent with $9$ vertices.  Hence, the above formula yields an upper bound of 45.  Using analytical methods we computed all possible maximal green sequences for this quiver.  We determined that the longest MGS for this quiver has length $35$. Notice that the thickened arrows between shared vertices are oriented in the same direction.  The quiver 
\begin{center} \begin{tikzpicture}[ 
> = stealth, 
            shorten > = 1pt, 
            auto,
            node distance = 1.5cm, 
            semithick 
        ]

        \tikzstyle{every state}=[
            draw = black,
            thick,
            fill = white,
            minimum size = 3mm
        ]
				\node[draw=none, fill=none] (1) {$1$};
        \node[draw=none, fill=none] (2) [above right of=1] {$2$};
        \node[draw=none, fill=none] (3) [right of=1] {$3$};
        \node[draw=none, fill=none] (4) [above right of=3] {$4$};
        \node[draw=none, fill=none] (5) [right of=3] {$5$};
        \node[draw=none, fill=none] (6) [above right of=5] {$6$};
        \node[draw=none, fill=none] (7) [right of=5] {$7$};
        \node[draw=none, fill=none] (8) [above right of=7] {$8$};
        \node[draw=none, fill=none] (9) [right of=7] {$9$};

        \path[->] (1) edge node { } (2);
				\path[->] (2) edge node { } (3);
				\path[->] (3) edge node { } (1);
				\path[->] (5) edge node { } (4);
				\path[->] (4) edge node { } (3);
				\path[->] (3) [ultra thick] edge node { } (5);
				\path[->] (5) edge node { } (6);
				\path[->] (6) edge node { } (7);
				\path[->] (7) [ultra thick] edge node { } (5);
				\path[->] (7) edge node { } (8);
				\path[->] (8) edge node { } (9);
				\path[->] (9) edge node { } (7);

\end{tikzpicture}
\end{center}

\noindent varies from the previous one in that the arrows between the three shared vertices alternate in orientation. The maximal length MGS for this quiver, also determined computationally, is $37$.

\end{ex}

\bibliographystyle{plain}

\end{document}